\documentclass[12pt]{amsart}
\usepackage{epsfig}
\usepackage{psfrag}
\usepackage{amsmath}
\usepackage{amssymb}
\usepackage{latexsym}
\usepackage{amscd}
\usepackage{amsfonts}
\usepackage{graphicx}

\usepackage{mathrsfs}
\usepackage{amsfonts}
\usepackage{amsopn}
\usepackage{amstext}
\usepackage{amscd}
\usepackage{amsthm}

\newtheorem{theorem}{Theorem}[section]

\newtheorem{definition}{Definition}[section]

\newtheorem{lemma}{Lemma}[section]

\newtheorem{proposition}{Proposition}[section]

\DeclareMathOperator{\conv}{conv}
\DeclareMathOperator{\expect}{E}
\DeclareMathOperator{\prob}{P}
\DeclareMathOperator{\variance}{VAR}

\DeclareMathOperator{\Ber}{Ber}
\DeclareMathOperator{\B}{B}
\DeclareMathOperator{\Prob}{P}

\DeclareMathOperator{\e}{e}

\DeclareMathOperator{\T}{T}
\DeclareMathOperator{\Tan}{T}
\DeclareMathOperator{\Nor}{N}

\DeclareMathOperator{\diff}{d}

\DeclareMathOperator{\Sphere}{S}
\DeclareMathOperator{\Span}{span}
\DeclareMathOperator{\Aff}{aff}
\DeclareMathOperator{\Normal}{N}

\DeclareMathOperator{\closure}{cl}

\newcommand{\R}{\mathbb R}
\newcommand{\N}{\mathbb N}

\newcommand{\comment}[1]{}

\newcommand{\Bigo}{\mathscr O}
\newcommand{\Cube}{Q}

\numberwithin{equation}{section}
\begin{document}
\title[Distribution of Optimally Aligned Letter Pairs]{Distribution of Aligned Letter Pairs in Optimal Alignments of Random Sequences}

\author[R.A.~Hauser]{Raphael Hauser} 
\address{Raphael Hauser, 
Mathematical Institute, University of Oxford, 24-29 St Giles', Oxford OX1 3LB, United Kingdom}
\email{hauser@maths.ox.ac.uk}
\thanks{Raphael Hauser was supported by the Engineering and Physical Sciences Research Council [grant number EP/H02686X/1]}
\author[H.F.~Matzinger]{Heinrich Matzinger}
\address{Heinrich Matzinger,
School of Mathematics, Georgia Institute of Technology, 686 Cherry Street, Atlanta, GA 30332-0160 USA}
\email{matzi@math.gatech.edu}
\thanks{Heinrich Matzinger was supported by the Engineering and Physical Sciences Research Council [grant number EP/I01893X/1], IMA Grant SGS29/11, and by Pembroke College Oxford}




\begin{abstract}Considering the optimal alignment of two i.i.d.\ random sequences of length $n$, we show that when the scoring function is chosen randomly, almost surely the empirical distribution of aligned letter pairs in all optimal alignments converges to a unique limiting distribution as $n$ tends to infinity. This result is interesting because it helps understanding the microscopic path structure of a special type of last passage percolation problem with correlated weights, an area of long-standing open problems. Characterizing the microscopic path structure yields furthermore a robust alternative to optimal alignment scores for testing the relatedness of genetic sequences.
\end{abstract}

\maketitle

\section{Introduction}

\subsection{Basic Definitions and Overview}\label{overview} Let $\mathcal{A}$ denote a finite alphabet, and let us consider alignments with gaps of two strings $x$ and $y$ of equal length $n$ consisting of letters from $\mathcal{A}$. For each such alignment $\pi$ we may count the number of letter pairs of different types aligned with each other and divide this number by $n$. Collecting these ratios for each possible pair of letters results in what we call the {\em empirical distribution vector} and denote by $\vec{p}_\pi(x,y)$. We remark that this results in the empirical probability distribution in the classical sense scaled by a factor $\tau\geq 1$ that is due to the presence of gaps. The classical distribution can of course be recovered by normalizing our notion of distribution. The set of all such empirical distribution vectors for given strings $x$ and $y$ will be denoted by 
\begin{equation*}
SET(x,y):=\left\{\vec{p}_\pi(x,y):\,\pi\text{ is an alignment with gaps of $x$ and $y$}\right\}.
\end{equation*}

{\footnotesize Let us give an example to illustrate the concepts we just introduced: Take $n=4$, consider the strings $x=\mathfrak{aabb}$ and $y=\mathfrak{abab}$, where the alphabet consists
of two letters $\mathcal{A}=\{\mathfrak{a,b}\}$, and let us look at a few alignments with gaps of $x$ and $y$. 
First, let $\pi$ be given by 
\begin{equation*}
\begin{array}{c||c|c|c|c|c}
x&\mathfrak{a}&\mathfrak{a}&\mathfrak{b}& &\mathfrak{b}\\\hline
y&\mathfrak{a}& &\mathfrak{b}&\mathfrak{a}&\mathfrak{b}
\end{array}
\end{equation*}

There is one $\mathfrak{a}$ aligned with $\mathfrak{a}$ and hence the coefficient
$p_{\mathfrak{aa}}=1/4$. Two pairs of letters $\mathfrak{b}$ are aligned with each other,
so $p_{\mathfrak{bb}}=2/4=0.5$. One letter $\mathfrak{a}$ from $x$ is aligned with a
gap, so that $p_{\mathfrak{aG}}=1/4$. Here and elsewhere we use the symbol $\mathfrak{G}$ for 
a gap. And finally, one letter $\mathfrak{a}$ from $y$ is aligned with a gap, so that
$p_{\mathfrak{Ga}}=1/4$. The empirical distribution vector of the alignment $\pi$ is now given by 
\begin{equation*}
\vec{p}_{\pi}(x,y)=\left(p_{\mathfrak{aa}},p_{\mathfrak{ab}},p_{\mathfrak{aG}},p_{\mathfrak{ba}},
p_{\mathfrak{bb}},p_{\mathfrak{bG}},p_{\mathfrak{Ga}},p_{\mathfrak{Gb}}\right)=
(0.25,0,0.25,0,0.5,0,0.25,0).
\end{equation*}

A second alignment $\nu$ is given by 
\begin{equation*}
\begin{array}{c||c|c|c|c}
x&\mathfrak{a}&\mathfrak{a}&\mathfrak{b}&\mathfrak{b}\\\hline
y&\mathfrak{a}&\mathfrak{b}&\mathfrak{a}&\mathfrak{b}
\end{array}
\end{equation*}
This time we find the following empirical distribution, 
\begin{equation*}
\vec{p}_\nu(x,y)=\left(p_{\mathfrak{aa}},p_{\mathfrak{ab}},p_{\mathfrak{aG}},
p_{\mathfrak{ba}},p_{\mathfrak{bb}},p_{\mathfrak{bG}},p_{\mathfrak{Ga}},p_{\mathfrak{Gb}}\right)=
(0.25,0.25,0,0.25,0.25,0,0,0).
\end{equation*}

Finally, consider the alignment $\mu$ given by 
\begin{equation*}
\begin{array}{c||c|c|c|c|c}
x&\mathfrak{a}&\mathfrak{a}&\mathfrak{b}& &\mathfrak{b}\\\hline
y&\mathfrak{a}&\mathfrak{b}& &\mathfrak{a}&\mathfrak{b}
\end{array},
\end{equation*}
The empirical distribution is given by 
\begin{equation*}
\vec{p}_\mu(x,y)=\left(p_{\mathfrak{aa}},p_{\mathfrak{ab}},p_{\mathfrak{aG}},
p_{\mathfrak{ba}},p_{\mathfrak{bb}},p_{\mathfrak{bG}},p_{\mathfrak{Ga}},p_{\mathfrak{Gb}}\right)=
(0.25,0.25,0,0,0.25,0.25,0.25,0).
\end{equation*}

Note that the coefficients of the empirical distribution vector do not usually add up to one. 
For example, in the case of alignment $\mu$, we get
\begin{equation*}
p_{\mathfrak{aa}}+p_{\mathfrak{ab}}+p_{\mathfrak{aG}}+p_{\mathfrak{ba}}+p_{\mathfrak{bb}}
+p_{\mathfrak{bG}}+p_{\mathfrak{Ga}}+p_{\mathfrak{Gb}}=1.25>1,
\end{equation*}
the reason being that we divided by 
the length $n$ of the strings instead of the number of columns of the alignment. If the alignment were without gaps, the coefficients would add up to one and represent frequencies, but in general this is not the case. However, the coefficients of the empirical distribution vector are proportional to the actual frequencies and hence may be thought of intuitively as representing these.}\\

Let us next consider two random strings $X=X_1\dots X_n$ and $Y=Y_1\dots Y_n$, where $X_i$ abd $Y_i$ are i.i.d.\ random variables taking values in the alphabet ${\mathcal A}$. Consider the set of all empirical distribution vectors that can be obtained from aligning $X$ with $Y$ and inserting the gaps in different places. We denote the convex hull of this set by 
$$SET^n:=\conv\left(SET(X, Y)\right).$$

One of our main results, Theorem \ref{convergence2SET}, will establish that 
\begin{equation*}
\lim_{n\rightarrow\infty}d(SET^n,SET)\stackrel{a.s.}{=}0,
\end{equation*}
where $SET$ is a unique limiting set that only depends on the distribution of the sequences $X$ and $Y$, 
but not their realization, and where $d(\cdot,\cdot)$ denotes the Hausdorff distance between two subsets in $\R^n$ and is defined in \eqref{set distance} below. 

Let $\mathcal{A}^*$ denote the alphabet $\mathcal{A}$ augmented by the symbol $\mathfrak{G}$, which stands for a gap, and consider functions $S$ from $\mathcal{A}^*\times\mathcal{A}^*$ into the set of real numbers. Such functions will be called {\em scoring functions}. For a gapped alignment $\pi$ of $x=x_1\dots x_n$ and  $y=y_1\dots y_n$, let us define the score $S_\pi(x,y)$ under a scoring function $S$ as the sum of scores of the aligned symbols pairs from $\mathcal{A}^*$. An alignment of $x$ and $y$ is called 
{\em optimal under $S$} if it maximizes $S_\pi(x,y)$ amongst all gapped alignments of $x$ and $y$.

Another main result of this paper, Theorem \ref{haupttheorem}, shows that when the scoring function $S$ is chosen at random, the empirical distribution of any optimal alignment of the random strings $X=X_1\dots X_n$ and $Y=Y_1\dots Y_n$ with respect to $S$ almost surely approaches a unique limiting vector $\vec{p}_S$ as $n$ tends to infinity. Apart from the realization of $S$, this limit vector only depends on the {\em distribution} of $X$ and $Y$, but not on their {\em realizations}. 

{\footnotesize Let us further illustrate the concept of {\em optimal alignment} of two strings by means of an example. Consider a scoring function $S$ that takes the value $1$ for identically aligned letters, and $0$ otherwise. In the context of the three examples of alignments introduced earlier, we find the following scores, 
\begin{align*}
S_\pi(x,y)&=S(\mathfrak{a,a})+S(\mathfrak{a,G})+S(\mathfrak{b,b})+S(\mathfrak{G,a})+S(\mathfrak{b,b})
=1+0+1+0+1=3,\\
S_\mu(x,y)&=S(\mathfrak{a,a})+S(\mathfrak{a,b})+S(\mathfrak{b,G})+S(\mathfrak{G,a})+S(\mathfrak{b,b})
=1+0+0+0+1=2,\\
S_\nu(x,y)&=S(\mathfrak{a,a})+S(\mathfrak{a,b})+S(\mathfrak{b,a})+S(\mathfrak{b,b})
=1+0+0+1=2.
\end{align*}
One could check that the score of $3$ is not exceeded for any alignment of $x$ and $y$ and hence find that $\pi$ maximizes the alignment score of $x$ and $y$ under $S$. We thus say that $\pi$ is an optimal alignment under $S$, although note that it may not be the unique alignment with this property. }

The alignment score can also be viewed as the value that an appropriately defined linear functional takes on the empirical distribution vector of an alignment: Let $f_S:\mathbb{R}^8\rightarrow \mathbb{R}$ be defined by
\begin{align}
f_S(\vec{p})&=f_S\left(p_{{\mathfrak aa}},p_{\mathfrak{ab}},p_{\mathfrak{aG}},p_{\mathfrak{ba}},p_{\mathfrak{bb}},p_{\mathfrak{bG}},
p_{\mathfrak{Ga}},p_{\mathfrak{Gb}}\right)
\label{defined in}\\
&=S(\mathfrak{a,a})p_{\mathfrak{aa}}+S(\mathfrak{a,b})p_{\mathfrak{ab}}
+S(\mathfrak{a,G})p_{\mathfrak{aG}}+S(\mathfrak{b,a})p_{\mathfrak{ba}}
+S(\mathfrak{b,b})p_{\mathfrak{bb}}\nonumber\\
&\hspace{2cm}+S(\mathfrak{b,G})p_{\mathfrak{bG}}+S(\mathfrak{G,a})p_{\mathfrak{Ga}}
+S(\mathfrak{G,b})p_{\mathfrak{G,b}}.\nonumber
\end{align}
It is then the case that 
\begin{equation}\label{f and S}
S_\pi(x,y)=n f_S\left(\vec{p}_\pi(x,y)\right)
\end{equation}
holds for any alignment $\pi$ of $x$ and $y$. 

Recall that $SET(x,y)$ is  defined as the set of all empirical distribution vectors of alignments of $x$ and $y$, 
\begin{equation*}
SET(x,y):=\left\{\vec{p}_\pi(x,y):\,\pi\text{ is an alignment of $x$ and $y$}\right\}.
\end{equation*}
In particular, in our current example $SET(x,y)=SET(\mathfrak{abab},\mathfrak{aabb})$ contains the vectors 
\begin{align*}
\vec{p}_\mu(x,y)=&
(0.25,0.25,0,0,0.25,0.25,0.25,0)\\
\vec{p}_\pi(x,y)=&
(0.25,0,0.25,0,0.5,0,0.25,0)\\
\vec{p}_\nu(x,y)=&
(0.25,0.25,0,0.25,0.25,0,0,0).
\end{align*}
Further, we write 
\begin{equation}\label{LS}
L_S(x,y):=\max_\pi S_{\pi}(x,y)
\end{equation}
for the optimal alignment score of $x$ and $y$, where the maximum is taken over all gapped alignments $\pi$ of $x$ with $y$. Thus, in our current example we have $L_S(\mathfrak{abab},\mathfrak{aabb})=3$. The rescaled maximum alignment score can also be seen as the maximum value taken by the functional $f_S$ over $SET(x,y)$, 
\begin{equation*}
\frac{L_S(x,y)}{n}= \max_{\vec{p}\in SET(x,y)}f_S(\vec{p}).
\end{equation*}

\subsection{Motivation}

\subsubsection{First passage Percolation} 
The problem of understanding the structure of an optimal path in first and last passage percolationhas was  recognized as being important several decades ago but still remains largely unresolved, see Kesten \cite{howard}. Consider the set of edges 
\begin{equation*}
E:=\left\{ \{(z,w),(z,w+1)\},\,\{(z,w),(z+1,w)\}:\,z,w\in\mathbb{Z}\right\}
\end{equation*}
of the integer lattice. $E$ thus consists of vertical and horizontal edges of unit length incident to points in $\R^2$ with integer coordinates. Consider a setup in which a random weight $w(e)$ is associated with each edge $e\in E$. In the classical setting of First Passage Percolation, these random weights are i.i.d.\ distributed, and a path of smallest total weight between two points $a$ and $b$ is sought. Any admissible path must consist of consecutive adjacent edges $e_1,e_2,\dots,e_n\in E$, and $e_1$ and $e_n$ must be incident to $a$ and $b$ respectively. The weights can also be interpreted as the time it takes to cross an edge, with the total weight $w(e_1)+w(e_2)+\dots+w(e_n)$ of the path corresponding to the passage time from $a$ to $b$ via the chosen path, and a minimum weight path corresponding to a fastest link between the two points. 

An example of an open problem relating to the microstructure of an optimal path in first passage percolation is the following (see Kesten \cite{howard}):  Consider the two points $a=(0,0)$ and $b=(0,n)$. What is the proportion of vertical and horizontal edges in an shortest path from the point $(0,0)$ to $(0,n)$, and does this proportion converge as $n$ goes to infinity? We shall now argue that these questions are closely related to the central problem of this paper, which is to find the limiting empirical distribution of the aligned letter pairs. Consider the set of oriented edges 
\begin{equation*}
E':=\left\{\left((z,w),(z,w+1)\right), \left((z,w),(z+1,w)\right), \left((z,w),(z+1,w+1)\right):\,z,w\in\mathbb{Z}\right\},
\end{equation*}
let $x=x_1\dots x_n$ and $y=y_1\dots y_n$ be strings of letters from the alphabet ${\mathcal A}$, let a scoring function $S:{\mathcal A}^*\times{\mathcal A}^*\rightarrow\R$ be given, and let the weight $w(e)$ of an edge of type $e=((z,w),(z+1,z+1))$ (a diagonal edge) be equal to the score obtained by aligning the letter $x_{z+1}$ with $y_{w+1}$, 
\begin{equation*}
w(e)=S\left(x_{z+1},y_{w+1}\right).
\end{equation*}
For edges of type $e=((z,w),(z+1,w))$ (horizontal edges), let the weight $w(e)$ be given by the score of aligning a gap with the letter $x_{z+1}$, 
\begin{equation*}
w(e)=S\left(x_{z+1},\mathfrak{G}\right).
\end{equation*}
Likewise, for vertical edges $e=((z,w),(z,w+1))$, let 
\begin{equation*}
w(e)=S\left(\mathfrak{G},y_{w+1}\right).
\end{equation*}
In this manner, the problem of aligning $x$ with $y$ in an optimally according to the scoring function $S$ becomes a Last  Passage Percolation problem. The optimal alignment score $S(x,y)$ equals the weight of the maximum weight path going from $(0,0)$ to $(n,n)$. An optimal path  $e_1e_2\dots e_m$, that is, a path of maximum total weight among those that follow oriented edges from $E'$ and link $(0,0)$ to $(n,n)$, defines an optimal alignment of $x$ with $y$ in the following fashion: For any diagonal edge $((z,w),(z+1,w+1))$ that lies along the path, align the letter $x_{z+1}$ with $y_{w+1}$. align all other letters with gaps. Now note that when we know the limit of the  empirical distribution vector of the aligned letter pairs, we also know the proportion of gaps on the long run. In other words, the limiting distribution of the aligned letter pairs yields the asymptotic  proportion of horizontal and vertical edges in the optimal path in $E'$. This information is the equivalent to  knowing the asymptotic proportion of vertical and horizontal edges in last passage percolation in a model where the distribution of edge weights is different, and where diagonal edges are present. The corresponding results for first passage percolation can be obtained by multiplying the edge weights by $-1$. 

\subsubsection{Computational Genomics} 
In computational genomics the alignment score is a maximum likelihood ratio to decide which alignment is the most likely association of sequences that diverged by evolution. Any gapped alignment of
two DNA or RNA strings $x=x_1\dots x_n$ and $y=y_1\dots y_n$ represents a hypothesis about the evolutionary history of the two species relative to one another. Assume that the strings $x$ and $y$ are sections of genetic sequences from two extant species with a common ancestor. Aligning $x_i$ with $y_j$ corresponds to the hypothesis that both letters descended from a particular letter in the corresponding section in the ancestor's genome. A letter aligned with a gap may correspond to a letter in the ancestral genome that disappeared in one of the two descendants or a new letter that appeared through mutation. The value $S(\mathfrak{a,b})$ of the scoring function at $\mathfrak{a,b}\in\mathcal{A}^*$ is equal to the logarithm of the probability that a letter from the ancestral genome evolved into a letter $\mathfrak{a}$ in one of the extant species and into a letter $\mathfrak{b}$ in the other, assuming that letters mutate independently of their neighbors\footnote{This is of course an inexact approximation of true mutation dynamics.}. Naturally, the scoring function depends on how long ago the two species got separated in the evolutionary tree. Given more time since separation, the probability of mutation increases, and as a result the scoring function also will look different.

In practice it can sometimes be difficult to determine whether two given sequences are related or not, since a good choice of the scoring function $S$ may not be known a priory in absence of a good estimate of the time since evolutionary divergence between the two extant species. Sections of DNA-sequences might also look similar because they have a similar distribution of amino acids rather than being related by direct evolution from a particular section of the ancestral genome. An approach to overcoming this problem is to relate sequences by the microscopic path structure of optimal alignments. Hirmo, Lember and Matzinger \cite{caseStudy} found that optimal alignments of related and non-related sequences have entirely different microscopic structures. Preliminary experiments showed that basing a relatedness test on path structure works at least as well as the widely used test based on the BLASTZ algorithm, which works with a sophisticated scoring function. The results of our paper will help to take this work further.

\subsection{Monte Carlo Simulation} 
An appealing approach to the investigation of asymptotic qualities of optimal alignments of random sequences is to use Monte Carlo simulation. However, to derive rigorous bounds one usually needs to know the variance $\variance(L_n(S))$ of the alignment score as a function of the length $n$ of the aligned random sequences $X_1\dots X_n$ and $Y_1\dots Y_n$, and this order of fluctuation is not yet well understood.  In the special case of the longest common subsequence problem with sequences consisting of i.i.d.\ $\Ber(p)$ variables, Chvatal and Sankoff \cite{Sankoff1} conjectured that $\variance(L_n(S))=o(n^{2/3})$ when $p=0.5$. Steele \cite{Steele86} later proved the bound $\variance(L_n(S))\leq 2p(1-p)n$. Waterman \cite{Waterman-estimation} asked the question of whether this bound can be improved and found that for $p<0.5$, simulations suggest that the dependence of $\variance(L_n(S))$ on $n$ is linear. Boutet de Monvel \cite{boutet} found that this also applies to the case $p=0.5$, although the linear growth only sets in for very large $n$. Lember and Matzinger \cite{VARTheta} gave a rigorous proof of the linear order $\variance(L_n(S))=\Theta(n)$ in the case where $p$ is very small. Their analysis was based on showing that the manipulation of randomly selecting a letter of specified type from one of the two sequences and changing it into another specified type has a positive biased effect on the optimal alignment score. In Section \ref{fluctuation}, we significantly extend the applicability of this result to general scoring functions and random sequences whose distributions are not highly asymmetric: Theorem \ref{matz2} yields a sufficient criterion under which the asymptotic order of fluctuation $\variance(L_n(S))=\Theta(n)$ holds.

\subsection{Summary of Main Results}
We now amend the notation introduced in \eqref{LS} and write 
\begin{equation}\label{new old}
L_n(S)=\max_{\pi} S_{\pi}(X,Y)
\end{equation}
for the optimal alignment score of the two i.i.d.\ strings $X=X_1,X_2,\dots,X_n$ and $Y=Y_1,Y_2,\dots Y_n$ with letters from the alphabet $\mathcal{A}$. Since $X$ and $Y$ are random, the maximum score $L_n(S)$ is a random variable, and we are interested in its dependence on the scoring function $S:{\mathcal A}^*\times{\mathcal A}^*\rightarrow\R$ in the asymptotic regime where the length $n$ of the strings tends to infinity. Let $\lambda_n(S)$ denote the rescaled expected alignment score, 
\begin{equation}\label{lambda S}
\lambda_n(S):=\frac{E[L_n(S)]}{n}.
\end{equation}
A simple subadditivity argument, see Chv\`atal \& Sankoff \cite{Sankoff1}, shows that 
\begin{align}
\lambda_n(S)&\leq\lambda_m(S),\quad\forall\,n\leq m\in\N,\label{increasing}\\
\lambda(S)&:=\lim_{n\rightarrow\infty}\frac{\expect\left[L_n(S)\right]}{n}\text{ exists},\label{chvatal}\\
\prob&\left[\lim_{n\rightarrow\infty}\frac{L_n(S)}{n}=\lambda(S)\right]=1.\label{convergence of lambda}
\end{align}
Furthermore, Alexander \cite{Alexander} showed that in the case of the longest common subsequence problem, the convergence is slower than the order $\sqrt{\ln n}/\sqrt{n}$, 
\begin{equation*}
\lambda(S)-\lambda_n(S)\geq C\frac{\sqrt{\ln n}}{\sqrt{n}}. 
\end{equation*}
In Lemma \ref{rateoflambdan} we show that a lower bound of the same order also exists. 
The exact value of the Chv\`atal-Sankoff constant $\lambda(S)$ is unknown even in the simplest cases, and Montecarlo simulations to obtain estimates of even moderate accuracy are quite involved \cite{martinezlcs, lcscurve}. Note that $\lambda(S)$ depends of course on the distribution of the random strings.

Let $H(S)$ designate the half space
\begin{equation}\label{H}
H(S):=\left\{\vec{x}\in\R^{|{\mathcal A}^*|^2}:\,f_S(\vec{x})\leq\lambda(S)\right\},
\end{equation}
where $f_S$ is defined by 
\begin{align*}
f_S:\R^{|{\mathcal A}^*|^2}&\rightarrow\R,\\
\vec{x}&\mapsto\sum_{\mathfrak{a}\in{\mathcal A}^*}\sum_{\mathfrak{b}\in{\mathcal A}^*}
S(\mathfrak{a,b})x_{\mathfrak{ab}}. 
\end{align*}
We also consider 
\begin{equation}\label{SET}
SET:=\cap_S H(S),
\end{equation}
where the intersection is taken over all scoring functions $S$. It is immediate from its definition that $SET$ is a closed convex set, and in Lemma \ref{lala} we will furthermore show that it is compact with nonempty interior. Recall also the notation 
\begin{equation*}
SET(X,Y):=\left\{\vec{p}_\pi(X,Y):\,\pi\text{ is an alignment with gaps of $X$ and $Y$}\right\}
\subset\R^{|{\mathcal A}^*|^2}
\end{equation*}
introduced earlier for the set of empirical distribution vectors for gapped alignments of $X$ and $Y$. Since $X$ and $Y$ are random strings, $SET(X,Y)$ is a random set. We denote the convex hull of this set by 
\begin{equation}\label{SET^n}
SET^n:=\conv\left(SET(X,Y)\right),
\end{equation}
where we account for the length $n$ of the random strings $X$ and $Y$ notationally because we are interested in the asymptotic behavior when $n$ tends to infinity. 

The Hausdorff distance \cite{hausdorff} between two sets $A,B\subset\R^n$ is defined as follows, 
where $\|\cdot\|$ denotes the Euclidean norm, 
\begin{align}
d(A,B)&=\max\left(\sup_{x\in A}d(x,B), \sup_{y\in B}d(A,y)\right)\label{set distance}\\
d(x,B)&=\inf\left\{\|x-y\|:\,y\in B\right\},\nonumber\\
d(A,y)&=\inf\left\{\|x-y\|:\,x\in A\right\}.\nonumber
\end{align}

Let us now discuss the main results of this paper: 
\begin{enumerate}
\item Theorem \ref{convergence2SET} will establish that the random set $SET^n$ almost surely converges to the deterministic set $SET$ in terms of Hausdorff distance. As a consequence, if one were to simulate the sequences $X_1\dots X_n$ and $Y_1\dots Y_n$ and compute the convex hull of the empirical distribution vectors of all their alignments with gaps, one would find a set that closely resembles $SET$, provided $n$ is large enough.
\item Theorem \ref{convergencepoint} will show that the empirical distributions of all optimal alignments of $X$ and $Y$ almost surely converge to a deterministic distribution as $n$ tends to infinity, on condition that the scoring function $S$ be chosen such that $f_S$ has a unique maximizer in $SET$. When this condition is met, we denote the unique maximiser by $\vec{p}_S$. The statement of the theorem then says that the probability that there exists an optimal alignment of $X$ and $Y$ with respect to $S$ with empirical distribution further away than $\epsilon>0$ from $\vec{p}_S$ is negatively exponentially small in $n$, where $\epsilon$ is an arbitrary small constant independent of $n$.
\item The condition of Theorem \ref{convergencepoint} is difficult to verify in practice, but Theorem \ref{convex3} shows that when the scoring function $S$ is chosen randomly, then the condition is almost surely met, that is, $f_S$ has a unique maximizer in $SET$ with probability $1$. As a corollary, we obtain Theorem \ref{haupttheorem}, which says that for almost all scoring function $S$ the empirical distributions of all optimal alignments of $X$ and $Y$ almost surely converge to a deterministic distribution. 
\item Theorem  \ref{haupttheorem} allows for the derivation of a sufficient criterion to guarantee that the order of fluctuation of the optimal alignment score is linear in the length $n$ of the aligned random strings $X_1\dots X_n$ and $Y_1\dots Y_n$. The sufficient criterion depends on the scoring function $S$ and on the unique maximizer $\vec{p}_S$ of $f_S$ over $SET$ and constitutes a practical tool 
in the design of a statistical test on the order of fluctuation of the optimal score, see \cite{lcsMontecarlo}. 
\end{enumerate}

\subsection{A Few Key Ideas} 
Let $S:{\mathcal A}^*\times{\mathcal A}^*\rightarrow\R$ be a scoring function, $x$ and $y$ two strings of length $n$ with letters from ${\mathcal A}$, and $\pi$ a gapped alignment of $x$ and $y$. By \eqref{f and S}, the optimal alignment score of $x$ and $y$ with respect to $S$ satisfies 
\begin{equation*}
\max_{\pi}\,\frac{S_{\pi}(x,y)}{n}=\max_{\vec{p}\in SET(x,y)}f_S(\vec{p}). 
\end{equation*}
Applied to the random strings $X=X_1\dots X_n$ and $Y=Y_1\dots Y_n$, we find 
\begin{equation}\label{preAAA}
\frac{L_n(S)}{n}=\max_{\vec{p}\in SET(X,Y)}f_S\left(\vec{p}\right).
\end{equation}
A crucial observation is now that the maximum of a linear function over a closed set in $\R^n$ equals the maximum of the given function over the convex hull of the given set. Since $SET^n=\conv\left(SET(X,Y)\right)$, \eqref{preAAA} implies
\begin{equation}\label{AAA}
\frac{L_n(S)}{n}=\max_{\vec{p}\in SET^n}f_S\left(\vec{p}\right).
\end{equation}
By \eqref{convergence of lambda}, $L_n(S)/n$ almost surely converges to a deterministic constant $\lambda(S)$ which was also used to define $SET$, see \eqref{H} and \eqref{SET}. The definition of $SET$ shows further that 
\begin{equation}\label{inequality}
\max_{\vec{p}\in SET}f_S\left(\vec{p}\right)\leq\lambda(S),
\end{equation}
and Lemma \ref{lala} d), proven below, shows that the inequality in \eqref{inequality} holds in fact at equality. Combined with \eqref{AAA}, this implies 
\begin{equation}\label{dual convergence}
\max_{\vec{p}\in SET^n} f_S\left(\vec{p}\right)\stackrel{n\rightarrow\infty}{\longrightarrow}
\max_{\vec{p}\in SET} f_S\left(\vec{p}\right)\quad\text{almost surely.}
\end{equation}

At a first pass it is illustrative to consider an approximate proof of Theorem \ref{haupttheorem} that is free of technical details relating to large deviations that will be necessary to render the proof rigorous. Proposition \ref{convex2}  and Theorems \ref{convex1} and \ref{convex3}, which will be proven in Section \ref{convex geometry}, provide the crucial insight: by \eqref{dual convergence}, the conditions of Theorem \ref{convex1} are approximately met for $C=SET$ and $C^n=SET^n$, $n\in\N$, and hence, it is plausible to argue that $SET^n\rightarrow SET$. We remark that in rendering this argument rigorous later, we will use the fact that $L_n(S)/n$ converges to $\lambda(S)$ at a rate on the order of $\ln(n)/\sqrt{n}$, which follows directly from the Azuma-Hoeffding Inequality, as we shall see in Section \ref{appendix}. The convergence of $SET^n$ to $SET$ occurs thus at the same rate. Choosing the scoring function $S$ at random is tantamount to choosing $f_S$ randomly, whence Theorem \ref{convex3} shows that the conditions of Proposition \ref{convex2} are satisfied. Theorem \ref{haupttheorem} thus follows.

\subsection{Some Key Difficulties}
Consider the optimal alignment score $L_n(S)$ for some scoring function $S$. $L_n(S)$ is a random variable, because it depends on the realization of the i.i.d.\ random strings $X=X_1\dots X_n$ and $Y=Y_1\dots Y_n$. 
However, it is easy to see that changing the realization of only one of these variables results in a change of $L_n(S)$ by at most the deterministic constant 
\begin{equation*}
\max_{c,d,e\in\mathcal{A}^*}|S(c,d)-S(c,e)|. 
\end{equation*}
See Lemma \ref{change} for details. 
One can therefore apply the Azuma-Hoeffding Inequality to find that, on a scale of $\sqrt{n}$, the tail of $L_n(S)$ decays at least quadratically exponentially fast, see Lemma \ref{Azuma}.  This powerful tool lends itself to an elegant analysis of the asymptotic convergence of the alignment score and its fluctuation. 

In contrast, analyzing the convergence of the empirical distribution of letter pairs in optimal alignments is much harder: upon changing the realization of one of the random letters, it has to be assumed a priori that the entire optimal alignment and hence the relative frequencies at which alignments occur have changed. As a consequence, the Azuma-Hoeffding Inequality cannot be applied directly. Luckily, it can be applied indirectly through the optimal alignment scores of different scoring functions at the cost of having to deal with additional technicalities.

A further key difficulty is that for the scoring functions $S$ under consideration, it is required that $f_S$ be maximized in only one point on $SET$. This condition would be met if $SET$ was known to be strictly convex everywhere, but this seems very difficult to verify in practice, since the exact shape of $SET$ is unknown: $SET$ corresponds to the asymptotic shape of the wet zone in the first/last passage percolation formualtion of our problem, and determining the shape of the corresponding zone in standard first passage percolation is a long-standing open problem in the general case. We get around this problem by showing that if the scoring function $S$ is chosen at random, then with probability one there exists a unique maximizer of $f_S$ on 
on $SET$, see Theorem \ref{convex3}. 

{\footnotesize We remark that convergence of the empirical distribution of aligned letter pairs may not hold when $S$ is not chosen randomly. For a counter-example, construct a scoring function that takes the value $1$ for pairs of identical letters, $0$ otherwise, and take ${\mathcal A}=\{0,1\}$ and $X_i$, $Y_j$ i.i.d.\ Bernoulli $\Ber(1/2)$ variables. The optimal alignment score then corresponds to the length of the longest common subsequence (LCS), and the asymptotic empirical distribution of optimally aligned letters is not unique: Write out the optimally aligned sequences with the introduced gaps and subdivide them into sections of length $3$, e.g., 
\begin{equation*}
\begin{smallmatrix} 1&0&0\\ 1&\mathfrak{G}&\mathfrak{G}\end{smallmatrix}
\Big|
\begin{smallmatrix} 0&1&\mathfrak{G}\\ 0&1&1 \end{smallmatrix}
\Big|
\begin{smallmatrix} \mathfrak{G}&0&0\\ 1&\mathfrak{G}&0 \end{smallmatrix}
\Big|
\begin{smallmatrix} 1&1&0\\ \mathfrak{G}&\mathfrak{G}&\mathfrak{G} \end{smallmatrix}
\Big|
\begin{smallmatrix} 0&0&0\\ 0&0&0 \end{smallmatrix}
\Big|
\dots
\end{equation*}
One observes empirically that a positive proportion of triplets is of the form 
\begin{equation*}
\begin{smallmatrix} 0&1&\mathfrak{G}\\ \mathfrak{G}&1&0 \end{smallmatrix}, 
\begin{smallmatrix} \mathfrak{G}&0&1\\ 1&0&\mathfrak{G} \end{smallmatrix},
\begin{smallmatrix} \mathfrak{G}&1&0\\ 0&1&\mathfrak{G} \end{smallmatrix} 
\text{ or }
\begin{smallmatrix} 1&0&\mathfrak{G}\\ \mathfrak{G}&0&1 \end{smallmatrix}. 
\end{equation*}
The first two correspond to the pattern $\begin{smallmatrix}0&1\end{smallmatrix}$ in $X$ being aligned with the pattern $\begin{smallmatrix}1&0\end{smallmatrix}$ in $Y$, and the last two to the inverted situation. Thus, the first triplet can be exchanged for the second, and the third for the fourth without affecting $L_n(S)$ nor the optimality of the alignment. The empirical distribution of optimally aligned letter pairs changes however, as weight is shifted from the pairing $(1,1)$ to the pairing $(0,0)$.} 

\section{Set Convergence of Empirical Distributions}\label{set convergence}

\begin{theorem}\label{convergence2SET}
Let $SET$ and $SET^n$ be the sets defined in \eqref{SET} and \eqref{SET^n}, and let $d$ denote the Hausdorff distance defined in \eqref{set distance}.  Then 
\begin{equation}
\label{raphael}
\prob\left[d(SET^n,SET)\stackrel{n\rightarrow\infty}{\longrightarrow}0\right]=1. 
\end{equation}
\end{theorem}

\begin{proof}
By the definition of $d$ in \eqref{set distance}, we need to prove the two identities 
\begin{align}
\prob\left[\max_{\vec{x}\in SET^n}d(\vec{x},SET)\stackrel{n\rightarrow\infty}{\longrightarrow}0\right]&=1,
\label{Iconvergence}\\
\prob\left[\max_{\vec{x}\in SET}d(\vec{x},SET^n)\stackrel{n\rightarrow\infty}{\longrightarrow}0\right]&=1.
\label{IIconvergence}
\end{align}

To prove Equation \eqref{Iconvergence}, we use Lemma \ref{finitehalfspace} which establishes that, given $\epsilon>0$, there exist finitely many scoring functions $S_1,\dots,S_k$ such that 
\begin{equation*}
\max_{\vec{x}\in\bigcap_{i=1}^k H(S_i)}\,d\left(\vec{x}, SET\right)\leq\epsilon,
\end{equation*}
where the half-spaces 
\begin{equation*}
H\left(S_i\right):=\left\{\vec{x}\in\R^{|{\mathcal A}^*|^2}:\,f_{S_i}(\vec{x})\leq\lambda(S_i)\right\},
\end{equation*}
are defined as in Equation \eqref{H}. Let $H_n^+(S_i)$ denote the shifted half-space 
\begin{equation*}
H_n^+\left(S_i\right)=\left\{\vec{x}:\, f_{S_i}\left(\vec{x}\right)\leq\lambda_n\left(S_i\right)+\frac{\ln(n)}{\sqrt{n}}\right\},
\end{equation*} 
and let us define the event 
\begin{equation*}
{\mathscr A}_n\left(S_i\right)=\left\{\omega\in\Omega:\,SET^n(\omega)\subset H_n^+\left(S_i\right)\right\},
\end{equation*}
where $\Omega$ is the probability space over which the random sequences $X$ and $Y$ are defined. Corollary \ref{Azuma etc} and its proof show that the Azuma-Hoeffding Inequality implies  
\begin{equation*}
\prob\left[{\mathscr A}_n\left(S_i\right)^c\right]\leq n^{-c_{S_i}\ln n}, 
\end{equation*}
where $c_{S_i}>0$ is a constant that does not depend on $n$, see Theorem \ref{Azuma etc}. It follows that 
\begin{equation*}
\prob\left[SET^n\subset\bigcap_{i=1}^k H_n^+(S_i)\right]\geq 
1-\sum_{i=1}^k n^{-c_{S_i}\ln n}\geq 1-n^{-c\ln n}, 
\end{equation*}
where $c>0$ is a constant independent of $n$. The series $\sum_{n}n^{-c\ln n}$ being convergent, the Borel-Cantelli Lemma implies that almost surely there exists $n_0\in\N$ such that 
\begin{equation*}
SET^n\subset \cap_{i=1}^k H_n^+(S_i),\quad\forall\,n\geq n_0.
\end{equation*}
By the definition of the $S_i$, this implies that for all $n\geq n_0$, we have 
\begin{equation}\label{yoyo}
\max_{\vec{x}\in SET^n}d(\vec{x},SET)\leq C\times\frac{\ln n}{\sqrt{n}}+\epsilon
\end{equation}
where $C>0$ is a constant independent of $n$. This implies that 
\begin{equation*}
\prob\left[\limsup_{n\rightarrow\infty}\max_{\vec{x}\in SET^n}d\left(\vec{x},SET\right)\leq\epsilon\right]=1,\quad\forall\,\epsilon>0.
\end{equation*}
Finally, since this is true for all $\epsilon$ rational, Equation \eqref{Iconvergence} follows. 

To prove Equation \eqref{IIconvergence}, we employ Theorem \ref{222} that establishes that for any given $\epsilon>0$, there exist points $\vec{x}_1,\dots,\vec{x}_k\in SET_{SE}$, a certain subset\footnote{See Section \ref{convex geometry} for the relevant theory.} of the set of extreme points of $SET$,  and chosen such that 
\begin{equation*}
d\left(\vec{x},\conv\left(\vec{x}_1,\dots,\vec{x}_k\right)\right)\leq\epsilon,\quad\forall\,\vec{x}\in SET. 
\end{equation*}
We now claim that for each $\vec{x}_i$ there almost surely exists a sequence of points $\vec{x}_{1,i},\vec{x}_{2,i},\vec{x}_{3,i},\dots$ such that $\vec{x}_{j,i}\in SET^j$ for all $j\in\N$ and 
\begin{equation*}
\limsup_{j\rightarrow \infty}\left\|\vec{x}_{j,i}-\vec{x}_i\right\|\leq\epsilon.
\end{equation*}
By the triangular inequality, our claim implies that almost surely it is the case that 
\begin{equation*}
\limsup_{n\rightarrow\infty} \max_{\vec{x}\in SET}d(\vec{x},SET^n)\leq 2\epsilon,\quad
\forall\,\epsilon>0, 
\end{equation*}
and since this is true for all $\epsilon$ rational, Equation \eqref{IIconvergence} follows. 

It remains to prove our claim. Proposition \ref{lala d)} shows that there exist scoring functions $S_{\vec{x}_i}$ and $S_1,\dots,S_{\ell}$ such that 
\begin{equation*}
\vec{x}_i\in C_i:=\left\{\vec{x}:\,f_{S_{\vec{x}_i}}\left(\vec{x}\right)\geq\lambda\left(S_{\vec{x}_i}\right),\, f_{S_j}\left(\vec{x}\right)\leq\lambda\left(S_j\right), (j=1,\dots,\ell)\right\}\subset\B_{\epsilon}\left(\vec{x}_i\right)
\end{equation*}
and $C_{n,i}$ compact for all $n\in\N$, where 
\begin{equation*}
C_{n,i}:=\left\{\vec{x}:\,f_{S_{\vec{x}_i}}\left(\vec{x}\right)\geq\lambda_n\left(S_{\vec{x}_i}\right)-\frac{\ln n}{\sqrt{n}},\, f_{S_j}\left(\vec{x}\right)\leq\lambda\left(S_j\right)+\frac{\ln n}{\sqrt{n}}, (j=1,\dots,\ell)\right\}.
\end{equation*}
Further, by \eqref{increasing}, the sets $C_{n,i}$ are nested, and by compactness and \eqref{chvatal}, we have 
\begin{equation}\label{dxi2}
\limsup_{n\rightarrow\infty}\,d\left(\vec{x}_i,C_{n,i}\right)\leq\epsilon. 
\end{equation}

We will now show that with high probability $C_{n,i}$ has a nonempty intersection with $SET^n$. 
Consider the events 
\begin{align*}
{{\mathscr B}}_{n,j}&:=\left\{\omega\in\Omega:\,\frac{L_n\left(S_j\right)}{n}\leq\lambda\left(S_j\right)+\frac{\ln n}{\sqrt{n}}\right\},\\
{\mathscr C}_{n,i}&:=\left\{\omega\in\Omega:\,\exists\,\vec{x}\in SET^n\text{ s.t. }f_{S_{\vec{x}_i}}\left(\vec{x}\right)\geq\lambda_n\left(S_{\vec{x}_i}\right)-\frac{\ln n}{\sqrt{n}}\right\}.
\end{align*}

By Theorem \ref{Azuma etc}, we have 
\begin{equation*}
\prob\left[{\mathscr B}_{n,j}^c\right]\leq n^{-K_{j}\ln n}
\end{equation*}
where $K_{j}>0$ is a constant that does not depend on $n$. Note also that Equation \eqref{AAA} implies 
\begin{equation*}
{\mathscr B}_{n,j}=\left\{\omega\in\Omega:\,f_{S_j}\left(\vec{x}\right)\leq\lambda\left(S_j\right)+\frac{\ln n}{\sqrt{n}},\,\forall\,\vec{x}\in SET^n\right\}. 
\end{equation*}
In conjunction with \eqref{increasing}, Theorem \ref{Azuma etc} further implies 
\begin{equation*}
\prob\left[{\mathscr C}_{n,i}^c\right]\leq n^{-\ln(n)}.
\end{equation*}
But note that when the events ${\mathscr C}_{n,i}$ and ${\mathscr B}_{n,1},\dots{\mathscr B}_{n,\ell}$ occur jointly, then $SET^n\cap C_{n,i}\neq\emptyset$ holds. The probability that the intersection is empty is thus bounded from above by 
\begin{equation*}
P\left[{\mathscr C}_{n,i}^c\right]+\sum_{i=1}^{\ell}\prob\left[{\mathscr B}_{n,i}^c\right]\leq (\ell+1)n^{-K\ln n},
\end{equation*}
where $K>0$ is a constant that does not depend on $n$. 

In view of the fact that the series 
\begin{equation*}
\sum_{n=1}^\infty(\ell+1)n^{-K\ln n}
\end{equation*}
converges, the Borel-Cantelli Lemma now implies that, almost surely, for all but a finite number of $n\in\N$ there  exists $x_{n,i}\in SET^n\cap C_{n,i}$. In the finitely many cases where $SET^n\cap C_{n,i}=\emptyset$ we can pick an arbitrary point $x_{n,i}\in SET^n$ to complete the sequence. In view of \eqref{dxi2}, we thus find that almost surely it is possible to construct a sequence $(x_{n,i})_{n\in\N}$ with the claimed properties. Hence, this settles the theorem.
\end{proof}

\section{Point Convergence}\label{point convergence}

So far we established that the empirical distributions of optimal alignments of random sequences under any  scoring function asymptotically lie in $SET$. We will now show that for a fixed, randomly chosen scoring function $S$, the empirical distributions of all optimal alignments of $X$ and $Y$ under $S$ converge to a unique point in $SET$. Recall the notation $\vec{p}_{\pi}(x,y)$ introduced in Section \ref{overview}, and let us write 
\begin{equation*}
SET^*(X,Y)=\left\{\vec{p}_{\pi}(X,Y):\,\pi\text{ is an optimal alignment of }X\text{ and }Y\right\}
\end{equation*}
for the set of empirical distributions corresponding to optimal alignments of $X=X_1\dots X_n$ and $Y=Y_1\dots Y_n$. Consider the event 
\begin{equation*}
{\mathscr D}_n(\vec{p},\epsilon):=\left\{\omega\in\Omega:\,SET^*\left(X(\omega),Y(\omega)\right)
\setminus\overline{\B}_{\epsilon}\left(\vec{p}\right)\neq\emptyset\right\}
\end{equation*}
that there exists an optimal alignment $\pi$ of $x=X(\omega)$ and 
$y=Y(\omega)$ under the scoring function $S$ such that $\|\vec{p}_{\pi}(x,y)-\vec{p}\|>\epsilon$. 

\begin{theorem}\label{convergencepoint}
Let $S$ be a scoring function such that the hyperplane
\begin{equation}\label{hyperplane}
\left\{\vec{x}:\, f_S\left(\vec{x}\right)=\lambda(S)\right\}
\end{equation}
intersects $SET$ in a unique point $\vec{p}_S$, and let $\epsilon>0$ be given. Then there exists
a constant $K_\epsilon$ such that for all $n\in\mathbb{N}$ it is true that  
\begin{equation*}
\prob\left[{\mathscr D}_n\left(\vec{p}_S, \epsilon\right)\right]\leq e^{-K_\epsilon n}.
\end{equation*}
Furthermore, $SET^*(X,Y)\rightarrow\{\vec{p}_S\}$ almost surely as $n$ tends to infinity. 
\end{theorem}

\begin{proof}
By Lemma \ref{lala}, $SET$ is a compact convex set with nonempty intersection with the hyperplane \eqref{hyperplane}, and by \eqref{inequality} all such intersection points are maximizers of the optimization problem $\max_{\vec{y}\in SET}\langle\vec{s},\vec{y}\rangle$, where $\vec{s}$ is the normalization of the vector representation of the linear functional $f_S$ defined by the scoring function. It follows that $\vec{p}_S$ satisfies Definition \ref{strict curvature} of a point of strict curvature of $SET$. Proposition \ref{lala d)} therefore implies that there exist finitely many scoring functions $S_1,S_2,\dots,S_k$ and  thresholds $\epsilon_0,\dots,\epsilon_k>0$ such that 
\begin{equation}
\left\{\vec{x}:\,f_S\left(\vec{x}\right)\geq\lambda(S)-\epsilon_0\right\}\cap\bigcap_{i=1}^k\left\{\vec{x}:\,f_{S_i}\left(\vec{x}\right)\leq\lambda\left(S_i\right)+\epsilon_i\right\}\subset\B_{\epsilon}\left(\vec{p}_S\right), 
\label{epsilonset}
\end{equation}

Consider now the events 
\begin{equation*}
{\mathscr E}_{n,i}:=\left\{\omega\in\Omega:\,SET^n\subset\left\{\vec{x}:\,f_{S_i}\left(\vec{x}\right)\leq \lambda\left(S_i\right)+\epsilon_i\right\}\right\}.
\end{equation*}
By \eqref{AAA} this is equivalent to requiring that the rescaled optimal alignment score $L_n(S_i)/n$ satisfy   $L_n(S_i)/n\leq\lambda(S_i)+\epsilon_i$. By Theorem \ref{Azuma etc} there exists $K_i>0$ such that 
\begin{equation}\label{london1}
\prob\left[{\mathscr E}_{n,i}\right]\geq 1-\e^{-K_i n}\quad\forall\,n. 
\end{equation}
Let us further define the event  
\begin{equation*}
{\mathscr E}_{n,0}:=\left\{\omega\in\Omega:\,SET^n\cap\left\{\vec{x}:\,f_{S_i}\left(\vec{x}\right)\geq \lambda(S)-\epsilon_0\right\}\neq\emptyset\right\},
\end{equation*}
which is the same as requiring that $L_n(S)/n$ exceed the value $\lambda(S)-\epsilon_0$. Corollary \ref{Azuma etc} once again shows that there exists $K_0>0$ such that 
\begin{equation}\label{london2}
\prob\left[{\mathscr E}_{n,0}\right]\geq 1-\e^{-K_0 n}\quad\forall\,n. 
\end{equation}

Combining all of the above, we now find ${\mathscr D}_n^c\subseteq\bigcup_{i=0}^k{\mathscr E}_{n,i}$, so that 
\begin{equation*}
\prob\left[{\mathscr D}_n\right]\leq\sum_{i=0}^k\prob\left[{\mathscr E}^{c}_{n,i}\right]\leq\sum_{i=0}^k\e^{-K_i n}
\leq\e^{-K_{\epsilon}n}
\end{equation*}
for some constant $K_{\epsilon}>0$, as claimed. 

The last statement follows from the Borel-Cantelli Lemma in a similar construction as in the proof of Theorem \ref{convergence2SET}. 
\end{proof}

The above theorem shows that if $\vec{p}_S$ is the only solution to $f_S(\vec{p})=\lambda_S$, then, denoting any optimal alignment of $X_1\dots X_n$ and $Y_1\dots Y_n$ with respect to $S$ by $\pi_n$, it is true a.s.\ that $\vec{p}_{\pi_n}(X_1\dots X_n,Y_1\dots Y_n)\rightarrow\vec{p}_S$. Note however that the convergence rate was not specified. Our convergence argument, which is based on the Azuma-Hoeffding Inequality -- see Theorem \ref{Azuma etc} -- could be made quantitative if a bound on the curvature of $SET$ at $\vec{p}_S$ were known. 

Our second and main result of this section shows that the above theorem generically applies. For this purpose we consider a scoring function $S$ that is chosen randomly in such a way that if $\vec{S}$ denotes the normalization of the vector representation of the linear functional $f_S$, then $\vec{S}$ has an absolutely continuous distribution with respect to the Hausdorff measure (or uniform measure) on the sphere. In this case we say that $S$ has absolutely continous distribution. 

\begin{theorem}\label{haupttheorem}
Let the scoring function $S$ be chosen randomly from an absolutely continuous distribution, and let $\pi_n$ denote any optimal alignment of $X_1\dots X_n$ with $Y_1\dots Y_n$. Then almost surely $\vec{p}_{\pi}(X,Y)$ converges to a unique empirical distribution. 
\end{theorem}

\begin{proof}
By Theorem \ref{convex3}, the conditions of Theorem \ref{convergencepoint} apply with probability 1. 
\end{proof}

\section{Fluctuation of the Optimal Alignment Score}\label{fluctuation}

Let $X=X_1\dots X_n$  and $Y=Y_1\dots Y_n$ be the random strings introduced earlier, let ${\mathfrak a}$ and ${\mathfrak b}$ be two distinct letters from the alphabet ${\mathcal A}$, and let us define a new random string $\tilde{X}=\tilde{X}_1\dots\tilde{X}_n$ via the following compound procedure:
\begin{enumerate}
\item sample a realization $x=x_1\dots x_n$ of $X$,
\item if ${\mathcal J}:=\{i:\,x_i={\mathfrak a}\}\neq\emptyset$, 
\begin{enumerate}
\item let $J$ be a random index defined on some probability space $(\tilde{\Omega},\tilde{\prob})$ and taking values with uniform distribution on ${\mathcal J}$, 
\item select a sample $j=J(\tilde{\omega})$, 
\item set $\tilde{x}_j={\mathfrak b}$ and $\tilde{x}_i=x_i$ for $i\neq j$, 
\end{enumerate}
\item else set $\tilde{x}=x$. 
\end{enumerate}
Note that the distribution of $\tilde{X}$ generally differs from the distribution of $X$, and that, while $X=X_1\dots X_n$ consists of the first $n$ letters of a random sequence $(X_i)_{i\in\N}$, the same cannot be said about $\tilde{X}$: we only ever sample (at most) one entry of $X$ realized in the form of an ${\mathfrak a}$, independently of $n$, so that the probability of any given index to be chosen diminishes as $n$ grows: 

The following result was proven by Lember and Matzinger \cite{VARTheta}, where we use the notation
\begin{equation*}
\tilde{L}_n(S):=\max_{\pi}\,S_{\pi}(\tilde{X},Y),
\end{equation*}
in analogy to the earlier introduced random variable $L_n(S)=\max_{\pi}\,S_{\pi}(X,Y)$, and where we write $f(n)=\Theta(n)$ if there exist constants $0<c_1<c_2$ such that $c_1 n\leq f(n)\leq c_2 n$ for all $n\in\N$. 

\begin{theorem}\label{matz}
Let the scoring function $S$ and the distribution of $X$ and $Y$ be chosen so that there 
exist parameters $\beta, \varepsilon>0$ for which 
\begin{equation*}
\prob\left[\expect_{\tilde{\prob}}\left[\tilde{L}_n(S)-L_n(S)\,\|\,X,Y\right]\geq\varepsilon\right]\geq 1-\e^{-\beta n},\quad\forall\,n\in\N. 
\end{equation*}
Then the order of fluctuation of the optimal alignment score is given by 
\begin{equation*}
\variance\left[L_n(S)\right]=\Theta(n).
\end{equation*}
\end{theorem}

Up until now, the criterion of Theorem \ref{matz} could only be verified in a few special cases.  
We will next see that when the scoring function satisfies the conditions of Theorem \ref{convergencepoint}, then the criterion of Theorem \ref{matz} can be reduced to a condition that solely depends on $\vec{p}_S$ and that can be verified by Montecarlo simulation to high confidence:

\begin{theorem}\label{matz2}
Let $S$ is such that $\{\vec{x}:\,f_S(\vec{x})=\lambda(S)\}$ intersects $S$ in a unique point $\vec{p}_S=(p_{{\mathfrak c}{\mathfrak d}})$ and such that there exist ${\mathfrak a},{\mathfrak b}\in {\mathcal A}$ for which it is the case that 
\begin{equation*}
\sum_{{\mathfrak c}\in \mathcal{A}^*}p_{{\mathfrak a}{\mathfrak c}}\left(S_{{\mathfrak b}{\mathfrak c}}-S_{{\mathfrak a}{\mathfrak c}}\right)>0. 
\end{equation*}
Then the order of fluctuation of the optimal alignment score is given by 
\begin{equation*}
\variance\left[L_n(S)\right]=\Theta(n).
\end{equation*}
\end{theorem}

\begin{proof}
Let ${\mathcal J}=\{i\in\{1,\dots,n\}:\,X_i={\mathfrak a}\}$,  $q_{{\mathfrak a}}=\prob[X_1={\mathfrak a}]$, and let us define the event 
\begin{equation*}
{\mathscr F}_n:=\left\{\omega\in\Omega:\,\frac{n}{|{\mathcal J}|}\geq\frac{1}{2q_{\mathfrak a}}\right\}. 
\end{equation*}
Since $X$ has i.i.d.\ entries, McDiarmid's Inequality -- see Lemma \ref{Azuma} below -- implies that for all $n\in\N$, 
\begin{equation}\label{macD}
\prob\left[{\mathscr F}_n\right]\geq1-\e^{-n\frac{q_{\mathfrak a}^2}{2}}.
\end{equation}

Next, let 
\begin{equation*}
\varepsilon:=\frac{1}{4 q_{{\mathfrak a}}}\left\langle\vec{p}_S, \left(S_{{\mathfrak b}{\mathfrak c}}-S_{{\mathfrak a}{\mathfrak c}}\right)_{\mathfrak c}\right\rangle:=\frac{1}{4 q_{{\mathfrak a}}}\sum_{{\mathfrak c}\in \mathcal{A}^*}p_{{\mathfrak a}{\mathfrak c}}\left(S_{{\mathfrak b}{\mathfrak c}}-S_{{\mathfrak a}{\mathfrak c}}\right),
\end{equation*}
where $q_{\mathfrak a}=\prob[X_1={\mathfrak a}]$. By continuity of inner products, there exists a $\delta>0$ so that for any $\vec{p}\in\B_{\delta}(\vec{p}_S)$, we have 
\begin{equation}\label{haho}
\frac{1}{2q_{{\mathfrak a}}}\left\langle\vec{p}, \left(S_{{\mathfrak b}{\mathfrak c}}-S_{{\mathfrak a}{\mathfrak c}}\right)_{\mathfrak c}\right\rangle
\geq\varepsilon.
\end{equation}
Recall now the notations $SET^*(X,Y)$ and ${\mathscr D}_n(\vec{p},\epsilon)$ introduced in Section \ref{point convergence}. Theorem \ref{convergencepoint} shows that there exists $K_{\delta}>0$ such that the probability that all optimal alignments of $X$ and $Y$ have empirical distributions that lie within a distance $\delta$ of $\vec{p}_S$ equals
\begin{equation}\label{eQ}
\prob\left[SET^*(X,Y)\subseteq\B_{\delta}\left(\vec{p}_S\right)\right]=
\prob\left[{\mathscr D}^c_n\left(\vec{p}_S,\delta\right)\right]\geq 1-\e^{-K_{\delta}n},\quad
\forall\, n\in\N.
\end{equation}

But when ${\mathscr D}^c_n(\vec{p}_S,\delta)$ occurs, then for any optimal alignment $\pi_n^*$ of $X$ and $Y$, \eqref{haho} holds with $\vec{p}=\vec{p}_{\pi_n^*}(X,Y)$. Denoting the components of $\vec{p}_{\pi^*_n}(X,Y)$ by $p^*_{{\mathfrak c}{\mathfrak d}}$, where $({\mathfrak c},{\mathfrak d})$ are pairs of letters from ${\mathcal A}^*$, we have 
\begin{align}
\prob&\left[\expect_{\tilde{\prob}}\left[\tilde{L}_n(S)-L_n(S)\,\|\,X,Y\right]\geq\varepsilon\,\Bigr\|\,{\mathscr D}^c_n\left(\vec{p}_S,\delta\right), {\mathscr F}_n\right]\nonumber\\
&\geq\prob\left[\expect_{\tilde{\prob}}\left[S_{\pi^*_n}(\tilde{X},Y)-S_{\pi^*_n}(X,Y)\,\|\,X,Y\right]\geq\varepsilon\,\Bigr\|\,{\mathscr D}^c_n\left(\vec{p}_S,\delta\right), {\mathscr F}_n\right]\nonumber\\
&=\prob\left[\frac{n}{|{\mathcal J}|}\sum_{{\mathfrak c}\in\mathcal{A}^*}p^*_{{\mathfrak a}{\mathfrak c}}\left(S_{{\mathfrak b}{\mathfrak c}}-S_{{\mathfrak a}{\mathfrak c}}\right)\geq\varepsilon\,\Bigr\|\,{\mathscr D}^c_n\left(\vec{p}_S,\delta\right), {\mathscr F}_n\right]\nonumber\\
&\stackrel{\eqref{haho}}{\geq}\prob\left[\frac{2nq_{{\mathfrak a}}}{|{\mathcal J}|}\varepsilon\geq\varepsilon\,\Bigr\|\,{\mathscr D}^c_n\left(\vec{p}_S,\delta\right), {\mathscr F}_n\right]=1\label{macMac}
\end{align}
Therefore, 
\begin{align*}
\prob&\left[\expect_{\tilde{\prob}}\left[\tilde{L}_n(S)-L_n(S)\,\|\,X,Y\right]\geq\varepsilon\right]\\
&\geq\prob\left[\expect_{\tilde{\prob}}\left[\tilde{L}_n(S)-L_n(S)\,\|\,X,Y\right]\geq\varepsilon\,\Bigr\|\,{\mathscr D}^c_n\left(\vec{p}_S,\delta\right), {\mathscr F}_n\right]\times
\prob\left[{\mathscr D}^c_n\left(\vec{p}_S,\delta\right), {\mathscr F}_n\right]\\
&\stackrel{\eqref{macD},\eqref{eQ},\eqref{macMac}}{\geq} 1-\e^{-n\frac{q_{\mathfrak a}^2}{2}}-\e^{-K_{\delta}n},\quad\forall\,n\in\N.
\end{align*}
Thus, the conditions of Theorem \ref{matz} are met for $\beta>0$ small enough, and the claimed order of fluctuation holds. 
\end{proof}

\section{Appendix: Convex Geometry}\label{convex geometry}

We will now present geometric results required in the analysis of earlier parts of this paper. $\Sphere^{n-1}$ will denote the unit sphere in $\R^n$, $\B_{\rho}(\vec{x})$ the Euclidean ball of radius $\rho$ around $\vec{x}\in\R^n$, $d$ the Hausdorff distance, $\conv(\cdot)$ the convex hull and $\closure(\cdot)$ the closure of a set in the canonical subspace topology inherited from $\R^n$. We say that a convex set $C\subset\R^n$ has dimension $k$ if its affine hull $\Aff(C)\subset\R^n$ has dimension $k$.

\begin{theorem}\label{convex3}
Let $C\subset\R^n$ be nonempty compact convex, and let $\vec{S}:\Omega\rightarrow\Sphere^{n-1}$ be a random vector that takes values in the unit sphere with uniform distribution, defined on some probability space $(\Omega,{\mathscr A}, \Prob)$. Then for almost all $\omega\in\Omega$, the optimization 
problem $\arg\max_{\vec{y}\in C}\langle\vec{S}(\omega), \vec{y}\rangle$ has a unique solution. 
\end{theorem}

\begin{proof}
Let us first consider the case where $C$ has nonempty interior. Upon a shift of $C$ we may assume without loss of generality that $\vec{0}$ lies in the interior of $C$. Then the polar of $C$, 
\begin{equation*}
C^{\circ}=\left\{\vec{w}\in\R^n:\,\left\langle\vec{w}, \vec{y}\right\rangle\leq 1,\,\forall\vec{y}\in C\right\},
\end{equation*} 
is also compact convex with nonempty interior. Seen as the claim of the theorem is invariant under positive scaling, we may further assume without loss of generality that 
\begin{equation*}
\B_{3}(\vec{0})\subset C^{\circ}\subset\B_{\varrho}(\vec{0}). 
\end{equation*}

Next, let $\vec{s}\in\Sphere^{n-1}$ be a given point on the unit sphere and consider the function 
\begin{align*}
\tau_{\vec{s}}:\Tan_{\vec{s}}\Sphere^{n-1}&\rightarrow\R,\\
\vec{w}&\mapsto\max\left\{\tau>0:\,\tau\vec{w}\in C^{\circ}\right\}
\end{align*}
defined on the tangent space at $\vec{s}$. We claim that $\tau_{\vec{s}}$ is Lipschitz continuous on a sufficiently small neighbourhood $\mathscr{V}_{\vec{s}}$ of $\vec{s}$ in $\Tan_{\vec{s}}\Sphere^{n-1}\cap\B_{2}(\vec{0})$. Let $\vec{w}_1,\vec{w}_2\in\Tan_{\vec{s}}\Sphere^{n-1}\cap\B_{2}(\vec{0})$ and $W=\Span\{\vec{w}_1, \vec{w}_2\}$. For $(i=1,2)$ we then have 
\begin{align}
&1\leq\left\|\vec{w}_i\right\|<2,\label{piervi}\\
&\tau_{\vec{s}}\left(\vec{w}_i\right)=\max\left\{\tau>0:\,\tau\vec{w}_i\in C^{\circ}\cap W\right\},\label{vtoroi}\\
&1<\tau_{\vec{s}}\left(\vec{w}_i\right)\left\|\vec{w}_i\right\|\leq\varrho.\label{treti}
\end{align}
By \eqref{vtoroi}, we may assume without loss of generality that $\R^n=W$ for the purposes of proving  $|\tau_{\vec{s}}(\vec{w}_1)-\tau_{\vec{s}}(\vec{w}_2)|\leq L\|\vec{w}_1-\vec{w}_2\|$ . We refer the reader to Figure \ref{tau} for an 
illustration of the geometric setup. 
\begin{figure}
\begin{center}
\psfrag{w1}{$\vec{w}_1$}
\psfrag{w2}{$\vec{w}_2$}
\psfrag{a}{$a$}
\psfrag{b}{$b$}
\psfrag{tw}{$\tau_{\vec{s}}(\vec{w}_1)\vec{w}_1$}
\scalebox{0.7}{\includegraphics{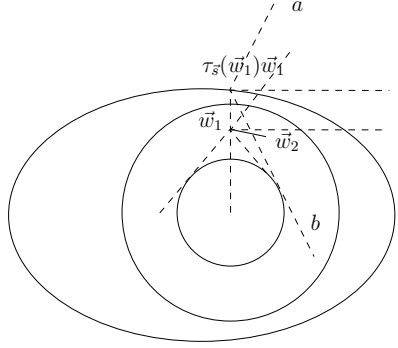}}
\end{center} 
\caption{The geometry of the Lipschitz estimate.} \label{tau}
\end{figure}
The lines $a$ and $b$ are the tangents from $\tau_{\vec{s}}(\vec{w}_1)\vec{w}_1$ to the unit sphere $\Sphere^1$ in $W$. Denote the angle between the line $\overline{\vec{w}_1\vec{w}_2}$ and the horizontal at $\vec{w}_1$ by $\theta$, the angle between the horizontal at $\tau_{\vec{s}}(\vec{w}_1)\vec{w}_1$ and the tangents $a$, $b$ by $\alpha$, and the angle between the horizontal at $\vec{w}_1$ and the two tangents from $\vec{w}_1$ to $\Sphere^1$ by $\beta$. Since the affine hull $\Aff(\vec{w}_1,\vec{w}_2)$ cannot enter $\B_1(\vec{0})$, it must lie wedged between the latter two tangents. In combination with \eqref{piervi}, this implies 
\begin{equation}\label{theta bound}
|\theta|\leq\beta=\frac{\pi}{2}-\arcsin\frac{1}{\left\|\vec{w}_1\right\|}\leq\frac{\pi}{2}-\arcsin\frac{1}{2}. 
\end{equation}
Further, \eqref{treti} implies 
\begin{equation}\label{alpha bound}
\alpha=\frac{\pi}{2}-\arcsin\frac{1}{\tau_{\vec{s}}\left(\vec{w}_1\right)\left\|\vec{w}_1\right\|}\leq
\frac{\pi}{2}-\arcsin\frac{1}{\varrho}. 
\end{equation}
Observe that, by convexity, the line segment between the point of tangency of $a$ at $\Sphere^1$ and $\tau_{\vec{s}}(\vec{w}_1)\vec{w}_1$ lies in $C^{\circ}$, and further that the definition of $\tau_{\vec{s}}(\vec{w}_1)$ implies $\tau_{\vec{s}}(\vec{w}_1)\vec{w}_1\in\partial C^{\circ}$. Therefore, the 
segment of $a$ above $\tau_{\vec{s}}(\vec{w}_1)\vec{w}_1$ lies outside $C^{\circ}$, and it follows that  
\begin{equation}\label{ratio inequality}
\frac{\|B\|}{\|C\|}\leq\tau_{\vec{s}}(\vec{w}_2)\leq\frac{\|A\|}{\|D\|},
\end{equation}
see Figure \ref{ratios}. 
\begin{figure}
\begin{center}
\psfrag{w1}{$\vec{w}_1$}
\psfrag{w2}{$\vec{w}_2$}
\psfrag{A}{$A$}
\psfrag{B}{$B$}
\psfrag{C}{$C$}
\psfrag{D}{$D$}
\psfrag{tw}{$\tau_{\vec{s}}(\vec{w}_1)\vec{w}_1$}
\scalebox{0.7}{\includegraphics{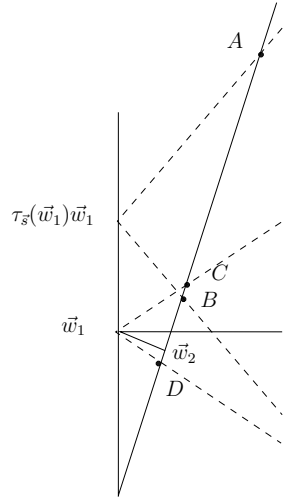}}
\end{center} 
\caption{Bounding $\tau_{\vec{s}}(\vec{w}_2)$ by ratios.} \label{ratios}
\end{figure}
Let $\varphi$ be the angle between $\vec{w}_1$ and $\vec{w}_2$, and let us assume $\varphi<(\pi -2\theta)/2$, so that the intersection points $A,B,C,D$ exist. This assumption is equivalent to limiting our analysis to a sufficiently small neighbourhood of $\vec{s}$ in $\T_{\vec{s}}\Sphere^{n-1}$, as assumed earlier. We can now express the inequalties \eqref{ratio inequality} in terms of the angles we introduced, 
\begin{equation*}
\tau_{\vec{s}}\left(\vec{w}_1\right)\frac{1-\tan\varphi\,\tan\beta}{1+\tan\varphi\,\tan\alpha}\leq\tau_{\vec{s}}\left(\vec{w}_2\right)\leq\tau_{\vec{s}}\left(\vec{w}_1\right)\frac{1+\tan\varphi\, \tan\beta}{1-\tan\varphi\,\tan\alpha}.
\end{equation*}
This can be simplified by Taylor expansion, 
\begin{align}
\left|\tau_{\vec{s}}\left(\vec{w}_2\right)-\tau_{\vec{s}}\left(\vec{w}_1\right)\right|&\leq\tau_{\vec{s}}\left(\vec{w}_1\right)\tan\varphi\left(\tan\alpha+\tan\beta\right)\nonumber\\
&\stackrel{\eqref{piervi},\eqref{treti},\eqref{theta bound},\eqref{alpha bound}}{\leq}
\varrho\tan\varphi\left(\varrho\sqrt{1-\varrho^{-2}}+\sqrt{3}\right),\label{cetviorti}
\end{align}
and since $\|\vec{w}_1-\vec{w}_2\|\geq \|\vec{w}_1\|\tan\varphi$, Equations \eqref{piervi} and \eqref{cetviorti} imply 
\begin{equation*}
\left|\tau_{\vec{s}}\left(\vec{w}_1\right)-\tau_{\vec{s}}\left(\vec{w}_2\right)\right|\leq L\left\|\vec{w}_1-\vec{w}_2\right\|,
\end{equation*}
with $L=\varrho(\varrho\sqrt{1-\varrho^{-2}}+\sqrt{3})$. 

Next, having shown that $\tau_{\vec{s}}$ is Lipschitz continuous on a sufficiently small open neighbourhood 
$\mathscr{V}_{\vec{s}}\subset\Tan_{\vec{s}}\Sphere^{n-1}$ of $\vec{s}$, Rademacher's Theorem \cite{Rademacher}
implies that $\tau_{\vec{s}}$ is Fr\'echet-differentiable everywhere on $\mathscr{V}_{\vec{s}}$ except on a 
null-set $\mathscr{B}_{\vec{s}}\subset\mathscr{V}_{\vec{s}}$. We now claim that if the optimization problem 
\begin{equation}
\vec{x}\left(\vec{s}\right)=\arg\max_{\vec{y}\in C}\left\langle\vec{s}, \vec{y}\right\rangle\label{problem unique}
\end{equation}
has multiple solutions, then $\tau_{\vec{s}}$ is G\^ateaux-nondifferentialble at $\vec{s}$. Since $\tau_{\vec{s}}$ is then also Fr\'echet nondifferentiable at $\vec{s}$, it must be the case that $\vec{s}\in\mathscr{B}_{\vec{s}}$. Let us thus suppose that \eqref{problem unique} has two different solutions, $\vec{x}_0\neq\vec{x}_1$. Then $\langle\vec{s},\vec{x}_1-\vec{x}_0\rangle=0$, so that we have $c_1:=\langle\vec{s},\vec{x}_1\rangle=\langle\vec{s},\vec{x}_0\rangle$. Furthermore, writing $c_2:=\langle \vec{x}_1-\vec{x}_0,\vec{x}_0\rangle$ and $c_3:=\langle\vec{x}_1-\vec{x}_0,\vec{x}_1\rangle$, our 
assumption that $\vec{x}_0\neq\vec{x}_1$ implies $c_2\neq c_3$. Without loss of generality we may assume that $c_2<c_3$. For all $\xi\in\R$ let us define $\vec{w}_{\xi}:=\vec{s}+\xi(\vec{x}_1-\vec{x}_0)$ and consider the restriction $\tau_{\vec{s}}|_{\vec{s}+\Span(\vec{x}_1-\vec{x}_0)}$ which we shall 
denote by $\tau(\xi):=\tau_{\vec{s}}(\vec{w}_{\xi})$. Clearly, if $\tau(\xi)$ is nondifferentiable at $\xi=0$, then $\tau_{\vec{s}}(\vec{w})$ is G\^ateaux-nondifferentiable at $\vec{w}=\vec{s}$. The definition of $\tau(\xi)$ implies $\langle\tau(\xi)\vec{w}_{\xi},\vec{x}_j\rangle\leq 1$ for $(j=0,1)$, so that 
\begin{align*}
\tau(\xi)&\leq\min\left(\frac{1}{c_1+c_2\xi},\frac{1}{c_1+c_3\xi}\right) \\
&=\frac{1}{c_1}\min\left(1-\frac{c_2}{c_1}\xi+\Bigo(\xi^2), 1-\frac{c_3}{c_1}\xi+\Bigo(\xi^2)\right). 
\end{align*}
Furthermore, we have $\tau(0)=1/c_1$. Therefore, 
\begin{equation*}
\frac{\diff}{\diff\xi_+}\tau(0)=\lim_{\xi\rightarrow 0+}\frac{\tau(\xi)-\frac{1}{c_1}}{\xi}\leq
-\frac{c_3}{c_1^2}<-\frac{c_2}{c_1^2}\leq\lim_{\xi\rightarrow 0-}\frac{\tau(\xi)-\frac{1}{c_1}}{\xi}=\frac{\diff}{\diff\xi_-}\tau(0),
\end{equation*}
showing that $\tau(\xi)$ is nondifferentiable at $\xi=0$, as claimed. 

Next, observe that $\tau_{\vec{s}}$ is Fr\'echet differentiable at $\vec{w}\in\mathscr{V}_{\vec{s}}$ if and 
only if the map 
\begin{align*}
\hat{\tau}:\,\Sphere^{n-1}&\rightarrow\R,\\
\vec{z}&\mapsto\max\left\{\tau>0:\,\tau\vec{z}\in C^{\circ}\right\}
\end{align*}
is differentiable at $\hat{w}:=\vec{w}/\|\vec{w}\|$ and if and only if $\tau_{\hat{w}}$ is differentiable 
at $\hat{w}$. Denoting the spherical projections of $\mathscr{V}_{\vec{s}}$ and $\mathscr{B}_{\vec{s}}$ by $\hat{\mathscr V}_{\vec{s}}$ and $\hat{\mathscr B}_{\vec{s}}$, the compactness of $\Sphere^{n-1}$ implies the existence of finitely many points $\vec{s}_1,\dots,\vec{s}_k\in\Sphere^{n-1}$ such that $\cup_{i=1}^k\mathscr{V}_{\vec{s}_i} = \Sphere^{n-1}$. Consequently, 
\begin{equation*}
\mathscr B=\cup_{i=1}^k \mathscr{B}_{\vec{s}_i}
\end{equation*}
is a null-set with the property that if Problem \eqref{problem unique} has multiple solutions for a given $\vec{s}\in\Sphere^{n-1}$, then $\vec{s}\in\mathscr{B}$. This proves the claim of the theorem in the case where $C$ has nonempty interior. 

Let us now consider the general case. When $C$ consists of a singleton, the claim of the theorem is trivial. We may thus assume that $\dim(C)\geq 1$. Upon a shift we may assume without loss of generality that $\vec{0}\in C$. Let $W=\Span(C)$ be the subspace spanned by $C$, and $W^{\perp}$ its orthogonal complement under the Euclidean inner product of $\R^n$. We denote the orthogonal projections onto these spaces by $\pi_W$ and $\pi_{W^{\perp}}$ respectively. Finally, let $\Sphere_W=\Sphere^{n-1}\cap W$ be the unit sphere in $W$, and 
\begin{align*}
\pi_{\Sphere}:\,\Sphere^{n-1}&\rightarrow\Sphere_W,\\
\vec{s}&\mapsto\frac{\pi_{W}\left(\vec{s}\right)}{\left\|\pi_{W}\left(\vec{s}\right)\right\|}
\end{align*}
the rescaled projection of $\Sphere^{n-1}$ into $W$. 

The condition $\dim(C)\geq 1$ implies $\dim(W^{\perp})\leq n-1$, and ${\mathscr B}_{W^\perp}=\{\omega\in\Omega:\,\vec{S}(\omega)\in W^{\perp}\}$ is a null-set. Hence,  $\pi_{\Sphere}(\vec{s})$ is defined for almost all $\vec{s}\in\Sphere^{n-1}$. Further, by isotropy of the uniform distribution on $\Sphere^{n-1}$, the random vector 
\begin{equation*}
\pi_{\Sphere}(\vec{S}):\Omega\setminus{\mathscr B}_{W^{\perp}}\rightarrow\Sphere_{W}
\end{equation*}
is uniformly distributed on $\Sphere_{W}$. Since $C$ has nonempty interior in the subspace topology of $W$, the case we already settled above applies and implies that 
\begin{equation*}
{\mathscr B}_{W}=\left\{\omega\in\Omega\setminus{\mathscr B}_{W^{\perp}}:\,\arg\max_{\vec{y}\in C}\left\langle\pi_{\Sphere}(\vec{S}(\omega)),\vec{y}\right\rangle\text{ is nonunique}\right\}
\end{equation*}
is a null-set. Observing that for $\vec{s}\in\Sphere^{n-1}\setminus W^{\perp}$ it is the case that 
\begin{equation*}
\arg\max_{\vec{y}\in C}\left\langle\vec{s},\vec{y}\right\rangle=\arg\max_{\vec{y}\in C}\left\langle\pi_{\Sphere}(\vec{S}(\omega)),\vec{y}\right\rangle, 
\end{equation*}
we find that $\arg\max_{\vec{y}\in C}\langle\vec{S}(\omega), \vec{y}\rangle$ has a unique solution if and only if $\omega$ is not in the null-set ${\mathscr B}={\mathscr B}_{W^{\perp}}\cup{\mathscr B}_{W}$. 
\end{proof}

The following notion will play a key role in the sequel. 

\begin{definition}\label{strict curvature}
Let $C\subset\R^n$ be convex compact. We say that a boundary point $\vec{x}\in\partial C$ is a point of strict curvature if there exists $\vec{s}\in\Sphere^{n-1}$ such that the optimization problem $\max_{\vec{y}\in C}\langle\vec{s},\vec{y}\rangle$ has $\vec{x}$ as unique maximizer. We denote the set of points of strict curvature by $C_{SE}$. 
\end{definition}

Note that if $C$ has a differentiable boundary, then any point where all principal curvatures are nonzero is a point of strict curvature. However, the set of points of strict curvature may be larger. For example, the epigraph of the curve $x\mapsto|x|^3$ has zero curvature at $x=0$, but this is a point of strict curvature nonetheless under our definition. Furthermore, Definition \ref{strict curvature} also applies to points where $\partial C$ is nondifferentiable and principal curvatures are not defined. For example, vertices of polytopes are points of strict curvature, while points on edges (1-faces) are not. Proposition \ref{lala d)} below also provides further intuition. 

For the purposes of the next result, let us recall that the normal cone of $C$ at $\vec{x}\in C$ is defined as follows, 
\begin{equation*}
\Nor_{\vec{x}}C=\left\{\vec{s}\in\R^n:\,\langle\vec{s}, \vec{x}-\vec{w}\rangle\geq 0,\,\forall\,\vec{w}\in C\right\}, 
\end{equation*}
or equivalently, 
\begin{align}
\Nor_{\vec{x}}C&=\bigl\{\vec{s}\in\R^n:\,\vec{x}=\arg\max_{\vec{w}\in C}\langle\vec{s}, \vec{w}\rangle\bigr\}\nonumber\\
&=\bigl\{\tau\vec{s}:\,\tau\geq 0,\,\vec{s}\in\Sphere^{n-1},\,\vec{x}=\arg\max_{\vec{w}\in C}\langle\vec{s},\vec{w}\rangle\bigr\}. 
\label{inspector hector}
\end{align}
By the dual description of $C$, it is the case that 
\begin{equation}\label{Norway}
\Nor_{\vec{x}}C\cap\Sphere^{n-1}\cap\Span(C)\neq\emptyset 
\end{equation}
if and only if $\vec{x}\in\partial C$, see also Lemma \ref{finitehalfspace} c). We denote the set of extreme points\footnote{A point $\vec{x}\in C$ is an extreme point of $C$ if it cannot be written as a convex combination of two points $\vec{y},\vec {z}\in C\setminus\{\vec{x}\}$.} of $C$ by $C_E$.

\begin{proposition}\label{lala d)}
For any $C\subset\R^n$ nonempty convex compact, the following hold true:
\begin{itemize}
\item[a) ] $\vec{x}\in C_{SE}$ if and only if there exists $\vec{s}\in\Nor_{\vec{x}}C\cap\Sphere^{n-1}$ and sequences $(\delta_k)_{k\in\N}, (\epsilon_k)_{k\in\N}\subset\R_+$ such that $\epsilon_k,\delta_k\rightarrow 0$ as $k$ tends to infinity and 
\begin{equation*}
\left\{\vec{y}\in C:\,\Nor_{\vec{y}}C\cap\Sphere^{n-1}\cap\B_{\delta_k}\left(\vec{s}\right)\neq\emptyset\right\}\subset\B_{\epsilon_k}\left(\vec{x}\right),\quad\forall\,k\in\N. 
\end{equation*}
\item[b) ] $C_{SE}\subseteq C_E\subseteq\closure(C_{SE})$. 
\item[c) ] $\{\vec{x}\in\partial C:\,\Normal_{\vec{x}} C\cap\Normal_{\vec{v}} C=\Span(C)^{\perp},\,\forall\,\vec{v}\in C\setminus\{\vec{x}\}\}\subset C_{SE}$. 
\item[d) ]  Let $\vec{x}_0\in C_{SE}$ and $\vec{s}_0\in\Nor_{\vec{x}_0} C\cap\Sphere^{n-1}$ be chosen such that $\vec{x}_0$ is the unique maximizer of $\max_{\vec{y}\in C}\langle\vec{s}_0,\vec{y}\rangle$, and let  $\epsilon>0$ be given. Then there exist finitely many points $\vec{x}_i \in C$ and normal vectors $\vec{s}_i\in\Nor_{\vec{x}_i} C\cap\Sphere^{n-1}$, $(i=1,\dots,k)$, such that
\begin{equation*}
C(\xi_0,\dots,\xi_k):=\left\{\vec{x}\in\R^n:\,\left\langle\vec{s}_0,\vec{x}-\vec{x}_0\right\rangle\geq \xi_0\right\}\cap\bigcap_{i=1}^k 
\left\{\vec{x}\in\R^n:\,\left\langle\vec{s}_i,\vec{x}-\vec{x}_i\right\rangle\leq \xi_i\right\}
\end{equation*}
is compact for all $(\xi_0,\dots,\xi_k)\in\R^{k+1}$, and $C(0,\dots,0)\subset\B_{\varepsilon}(\vec{x}_0)$.
\end{itemize}
\end{proposition}

\begin{proof}
a) Let $\vec{x}\in C_{SE}$, and let $\vec{s}\in\Nor_{\vec{x}} C\cap\Sphere^{n-1}$ such that $\vec{x}$ is the unique maximizer of 
\begin{equation}\label{cateye}
\max_{\vec{y}\in C}\langle\vec{s},\vec{y}\rangle. 
\end{equation}
Let $(\delta_k)_{k\in\N}\subset\R_+$ be a sequence such that $\delta_k\rightarrow 0$. We claim that there exists a sequence $(\epsilon_k)_{k\in\N}$ with the required properties. Supposing the claim to be wrong, there exists an $\epsilon>0$ and sequences $(\vec{x}_k)_{k\in\N}$, $(\vec{s}_k)_{k\in\N}$ such that 
\begin{align*}
\vec{x}_k&\in C\setminus\B_{\epsilon}\left(\vec{x}\right),\\
\vec{s}_k&\in\Nor_{\vec{x}_k} C\cap\Sphere^{n-1}\cap\B_{\delta_k}\left(\vec{s}\right).
\end{align*}
Since $\delta_k\rightarrow 0$, we have $\vec{s}_k\rightarrow\vec{s}$, and $C\setminus\B_{\epsilon}(\vec{x})$ being compact, we may assume without loss of generality that $\vec{x}_k\rightarrow\vec{x}_*$ for some $\vec{x}_*\in C\setminus\B_{\epsilon}(\vec{x})$. By \eqref{inspector hector}, this shows that $\vec{x}$ is not the unique maximizer of \eqref{cateye} and contradicts the membership of $\vec{x}$ in $C_{SE}$. Our claim is thus correct, and this establishes the ``only if'' part. 

To prove the converse, let us assume that $(\delta_k)_{k\in\N}$ and $(\epsilon_k)_{k\in\N}$ with the required properties exist. By \eqref{inspector hector}, all maximizers $\vec{x}_*$ of \eqref{cateye} must satisfy $\vec{s}\in\Nor_{\vec{x}_*}C$,  and since $\vec{s}\in\Sphere^{n-1}\cap\B_{\delta_k}(\vec{s})$, this implies $\vec{x}_*\in\B_{\epsilon_k}(\vec{x})$  for all $k\in\N$, which can only be true if $\vec{x}_*=\vec{x}$. This shows that $\vec{x}$ is the unique maximizer of \eqref{cateye}, and hence $\vec{x}\in C_{SE}$. 

b) For any point $\vec{x}\in C\setminus C_E$ there exist two other points $\vec{v},\vec{w}\in C$ of which $\vec{x}$ is a convex combination, $\vec{x}=\xi\vec{v}+(1-\xi)\vec{w}$. Let $\vec{s}\in\N_{\vec{x}} C\cap\Sphere^{n-1}$, so that $\vec{x}$ is a maximizer of \eqref{cateye}. The existence of such an $\vec{s}$ is guaranteed by \eqref{inspector hector}. By convexity, we have $\langle\vec{s},\vec{x}\rangle\leq\max(\langle\vec{s},\vec{v}\rangle, \langle\vec{s},\vec{w}\rangle)$. Without loss of generality, we may assume that the maximum is achieved at $\langle\vec{s},\vec{w}\rangle$, and therefore, $\vec{w}$ is also a maximizer of \eqref{cateye}.  Since this construction works for any $\vec{s}\in\N_{\vec{x}} C\cap\Sphere^{n-1}$, this shows that $\vec{x}\notin C_{SE}$, and hence, $C_{SE}\subseteq C_E$. 

Next, we claim that $K_1:=\closure(\conv(C_{SE}))=C$. Assuming our claim to be wrong, there exists $\vec{x}\in C\setminus K_1$, and since $K_1\subset C$ is compact, the Hahn-Banach Separation Theorem then implies that there exists $\vec{s}\in\Sphere^{n-1}$ such that 
\begin{equation}\label{separation}
\eta:=\max_{\vec{y}\in K_1}\left\langle\vec{s},\vec{y}\right\rangle<\xi:=\left\langle\vec{s},\vec{x}\right\rangle. 
\end{equation}
By Theorem \ref{convex3}, there exist sequences $(\vec{s}_k)_{k\in\N}\subset\Sphere^{n-1}$ and $(\vec{y}_{k})_{k\in\N}\subset C_{SE}$ such that 
\begin{align*}
\vec{s}_k&\stackrel{k\rightarrow\infty}{\longrightarrow}\vec{s},\\
\vec{y}_k&=\arg\max_{\vec{y}\in C}\left\langle\vec{s},\vec{y}\right\rangle,
\end{align*}
and since $K_1$ is compact, we may further assume without loss of generality that $\vec{y}_k\rightarrow\vec{y}_*\in K_1$. Using the continuity of the function $\vec{s}\mapsto\max_{\vec{y}\in C}\langle\vec{s},\vec{y}\rangle$, we thus find that 
\begin{equation*}
\xi=\lim_{k\rightarrow\infty}\left\langle\vec{s}_k,\vec{y}_k\right\rangle=\lim_{k\rightarrow\infty}\left\langle\vec{s}_k,\vec{y}_*\right\rangle\leq\eta. 
\end{equation*}
This contradicts \eqref{separation} and proves our claim to be true. 

Next, let $\vec{x}\in C_E$. By what we know so far, there exists a sequence of points $(\vec{x}_k)_{k\in\N}\subset\conv(C_{SE})$ that converges to $\vec{x}$. The Carath\'eodory-Steinitz  Theorem \cite{steinitz} then implies that each $\vec{x}_k$ can be written as a convex combination 
\begin{equation*}
\vec{x}_k=\sum_{i=1}^{\dim(C)+1}\theta^k_{i}\vec{z}_{i,k}
\end{equation*}
of at most $\dim(C)+1$ points $\vec{z}_{i,k}\in C_{SE}$, $(i=1,\dots,\dim(C)+1)$. By compactness of 
\begin{equation*}
\Delta_{\dim(C)+1}\otimes\left(\otimes_{i=1}^{\dim(C)+1}\closure\left(C_{SE}\right)\right),
\end{equation*}
where $\Delta_{\dim(C)+1}$ denotes the $\dim(C)$-dimensional simplex, we may assume without loss of generality that 
\begin{align*}
\theta^k&\stackrel{k\rightarrow\infty}{\longrightarrow}\theta^*\in\Delta_{\dim(C)+1},\\
\vec{z}_{i,k}&\stackrel{k\rightarrow\infty}{\longrightarrow}\vec{z}_{i,*}\in\closure(C_{SE}), 
\end{align*}
so that 
\begin{equation*}
\vec{x}=\lim_{k\rightarrow\infty}\vec{x}_k=\sum_{i=1}^{\dim(C)+1}\theta^*_i\vec{z}_{i,*}\in
\conv(\closure(C_{SE}))\subset C=\closure(\conv(C_{SE}). 
\end{equation*}
Seen as $\vec{x}$ is an extreme point of $C$, it must also be an extreme point of the subset $\conv(\closure(C_{SE}))$. Hence, all $\vec{z}_{i,*}$ must be identical. This shows that $\vec{x}=\vec{z}_{1,*}\in\closure(C_{SE})$ and proves the inclusion $C_E\subseteq\closure(C_{SE})$. 

c) This follows directly from Equations \eqref{inspector hector} and \eqref{Norway}.

d) By the dual description of $C$, we have 
\begin{equation*}
\bigcap_{\vec{x}\in C}\bigcap_{\vec{s}\in\Nor_{\vec{x}}C\cap\Sphere^{n-1}}\left\{\vec{y}\in\R^n:\,
\left\langle\vec{s},\vec{y}-\vec{x}\right\rangle\leq 0\right\}=C, 
\end{equation*}
and since it suffices to take this intersection over a dense subset of $\vec{x}\in C$, we have 
\begin{equation}\label{eis}
\bigcap_{\vec{x}\in\partial C\setminus\{\vec{x}_0\}}\bigcap_{\vec{s}\in\Nor_{\vec{x}}C\cap\Sphere^{n-1}}\left\{\vec{y}\in\R^n:\,
\left\langle\vec{s},\vec{y}-\vec{x}\right\rangle\leq 0\right\}=C. 
\end{equation}
By compactness of $C$ there exists $\rho>0$ such that $C\subset\closure(\B_{\rho}(\vec{0}))$. Consider the compact set $K_2=\closure(\B_{\rho}(\vec{0}))\setminus\B_{\varepsilon}(\vec{x}_0)$. 
By the assumed properties of $\vec{x}_0$ and $\vec{s}_0$, it is further true that  
\begin{equation}
\left\{\vec{y}\in\R^n:\,\left\langle\vec{s}_0,\vec{y}-\vec{x}_0\right\rangle\geq 0\right\}\cap C=\left\{\vec{x}_0\right\}. \label{dru}
\end{equation}
Equations \eqref{eis} and \eqref{dru} now show that 
\begin{equation*}
K_2\subset\left\{\vec{y}\in\R^n:\,\left\langle\vec{s}_0, \vec{y}-\vec{x}_0\right\rangle<0\right\}\cup\bigcup_{\vec{x}\in C\setminus\{\vec{x}_0\}}
\bigcup_{\vec{s}\in\Nor_{\vec{x}}C\cap\Sphere^{n-1}}\left\{\vec{y}\in\R^n:\,\left\langle\vec{s}, \vec{y}-\vec{x}\right\rangle>0\right\}. 
\end{equation*}
By compactness of $K_2$, there exist finitely many $\vec{x}_i\in C\setminus\{\vec{x}_0\}$ and 
$\vec{s}_i\in\Nor_{\vec{x}_i} C\cap\Sphere^{n-1}$, $(i=1,\dots,k)$, such that 
\begin{equation*}
K_2\subset\left\{\vec{y}\in\R^n:\,\left\langle\vec{s}_0, \vec{y}-\vec{x}_0\right\rangle<0\right\}\cup\bigcup_{i=1}^k 
\left\{\vec{y}\in\R^n:\,\left\langle\vec{s}_i, \vec{y}-\vec{x}_i\right\rangle>0\right\}. 
\end{equation*}
Let $\vec{e}_j$ be the $j$-th coordinate vector in $\R^n$. Let us write $\vec{s}_{j,\ell}=(-1)^{\ell}\vec{e}_j$ for $\ell\in\{0,1\}$, and let  
\begin{equation*}
\vec{x}_{j,\ell}\in\arg\max_{\vec{w}\in C}\left\langle\vec{s}_{j,\ell},\vec{w}\right\rangle. 
\end{equation*}
Then $\vec{s}_{j,\ell}\in\Nor_{\vec{x}_{j,\ell}}C$, and 
\begin{equation*}
\Cube:=\bigcap_{\ell=0}^1\bigcap_{j=1}^n\left\{\vec{y}\in\R^n:\,\left\langle\vec{s}_{j,\ell}, \vec{y}-\vec{x}_{j,\ell}\right\rangle\leq 0\right\}
\end{equation*}
is a cuboid. By choosing $\rho$ large enough for $\Cube\subset\closure(\B_{\rho}(\vec{0}))$ to hold, 
and by including the $(\vec{x}_{j,\ell},\vec{s}_{j,\ell})$ among our list of points $\{(\vec{x}_i,\vec{s}_i):\,i=1,\dots,k\}$ if necessary, we can guarantee that 
\begin{equation*}
\left\{\vec{y}\in\R^n:\,\left\langle\vec{s}_0, \vec{y}-\vec{x}_0\right\rangle\geq 0\right\}\cap\bigcap_{i=1}^k\left\{\vec{y}\in\R^n:\,\left\langle\vec{s}_i, \vec{y}-\vec{x}_i\right\rangle\leq 0\right\}\subset\Cube\setminus K_2\subset
\B_{\varepsilon}(\vec{x}_0),
\end{equation*}
and that $C(\xi_0,\dots,\xi_k)$ is compact (although possibly empty) for all $(\xi_0,\dots,\xi_k)$. 
\end{proof}

\begin{proposition}\label{convex2}
Let $C\subset\R^n$ be nonempty convex compact  and $C^1,C^2,\dots$ compact subsets of $\R^n$ such that $d(C^n, C)\rightarrow 0$, where $d$ denotes the Hausdorff distance. Let $\vec{s}\in\Sphere^{n-1}$ be such that the optimization problem 
\begin{equation}\label{maximizer of}
\vec{x}_*=\arg\max_{\vec{x}\in C}\left\langle\vec{s}, \vec{x}\right\rangle 
\end{equation}
has a unique solution. For all $n\in\N$ let $\vec{x}_n$ be a solution of 
\begin{equation*}
\vec{x}_n=\arg\max_{\vec{x}\in C^n}\left\langle\vec{s}, \vec{x}\right\rangle.
\end{equation*}
Then $\vec{x}_n\rightarrow\vec{x}_*$ as $n$ tends to infinity. 
\end{proposition} 

\begin{proof}
Let $f(\vec{x})=\langle\vec{s},\vec{x}\rangle$ be the linear functional defined by $\vec{s}$. We note that 
\begin{equation*}
\liminf_{n\rightarrow\infty}\max_{\vec{y}\in C^n}f(\vec{y})\geq f(\vec{x}_*), 
\end{equation*}
since for any $\epsilon>0$ there exists $n_{\epsilon}$ such that for all $n\geq n_{\epsilon}$ there exists $\vec{y}_n\in A^n\cap\B_{\epsilon}(\vec{x}_*)$, and we have 
\begin{equation}\label{nepsilon}
f(\vec{x}_*)=f(\vec{y}_n)+\langle\vec{s},\vec{x}_*-\vec{y}_n\rangle<f(\vec{y}_n)+\epsilon. 
\end{equation}
Next, let $W=\Span(C)$ and $W^{\perp}$ be its orthogonal complement in $\R^n$. Upon shifting and rescaling, we may assume without loss of generality that $\B_{\rho}(\vec{0})\cap W\subset C\subset\B_1(\vec{0})\cap W$ for some $\rho>0$, so that $f(\vec{x}_*)\leq 1$. The condition $d(C^n, C)\rightarrow 0$ then implies that for any $\epsilon>0$ we may take $n_{\epsilon}$ to be large enough so that for $n\geq n_{\epsilon}$, 
\begin{equation*}
C^n\subset\left(1+\frac{\epsilon}{\rho}\right)C\times\left(\B_{\epsilon}(\vec{0})\cap W^{\perp}\right). 
\end{equation*}
Convexity of $C$ and the uniqueness of $\vec{x}_*$ as a maximizer of \eqref{maximizer of} imply that there exists $\delta(\epsilon)>0$ such that 
\begin{equation}\label{chanterelle}
f(\vec{x})<f(\vec{x}_*)-\left(2+\frac{f(\vec{x}_*)}{\rho}\right)\epsilon
\end{equation}
for all $\vec{x}\in C\setminus\B_{\delta(\epsilon)}(\vec{x}_*)$, and that $\delta(\epsilon)\rightarrow 0$ when $\epsilon\rightarrow 0$. Furthermore, we have 
\begin{equation}\label{parasol}
\sup_{y\in\B_{\epsilon}(\vec{0})\cap W^{\perp}}f(\vec{y})<\epsilon. 
\end{equation}
\eqref{chanterelle} and \eqref{parasol} imply that for all $\vec{x}\in (1+\epsilon/\rho)\left(C\setminus\B_{\delta(\epsilon)}(\vec{x}_*)\right)\times\left(\B_{\epsilon}(\vec{0})\cap W^{\perp}\right)$, 
\begin{align*}
f(\vec{x})&<\left(1+\frac{\epsilon}{\rho}\right)\left(f(\vec{x}_*)-\left(2+\frac{f(\vec{x}_*)}{\rho}\right)\epsilon
\right)+\epsilon\\
&<f\left(\vec{x}_*\right)-\epsilon. 
\end{align*}
In particular, this applies to all $\vec{x}\in C^n\setminus(1+\epsilon/\rho)\B_{\delta(\epsilon)}(\vec{x}_*)$, and \eqref{nepsilon} shows that for $n\geq n_{\epsilon}$ we have  $\vec{x}_n\in(1+\epsilon/\rho)\B_{\delta(\epsilon)}(\vec{x}_*)$. We conclude that 
\begin{align*}
\|\vec{x}_n-\vec{x}_*\|&\leq\left\|\vec{x}_n-\left(1+\frac{\epsilon}{\rho}\right)\vec{x}_*\right\|+\left\|\left(1+\frac{\epsilon}{\rho}\right)\vec{x}_*-\vec{x}_*\right\|\\
&\leq\left(1+\frac{\epsilon}{\rho}\right)\delta(\epsilon)+\frac{\epsilon}{\rho}. 
\end{align*}
Since we may choose $\epsilon,\delta\rightarrow 0$ when $n$ is allowed to go to infinity, the result follows. 
\end{proof}

Next, we shall investigate the approximability of compact convex sets by polyhedra and polytopes. Results on outer approximations by polyhedra and algorithms to achieve this in practice are widespread in the literature on the cutting plane approach in numerical optimization, see e.g.\ Bertsekas \cite{bertsekas}. Similar results for inner approximations by polytopes play a key role in Markov chain Monte Carlo methods for the estimation of the volume of high dimensional convex bodies, see e.g.\ Jerrum \cite{jerrum}. The literature in both areas is focused on algorithms and relies on separation or membership oracles. As a result, the constructions use outer approximations by cutting planes that do not necessarily touch the boundary of the convex body to be approximated, and likewise, inner approximations use generators that generally do not lie on the boundary either. 

In contrast, the approximations required by our analysis have a crucial interplay with the boundary. For outer approximations, we would like cutting hyperplanes to be supported at points of strict curvature. Likewise, we would like inner approximations to be generated as the convex hull of points of strict curvature. Since we are not aware of such resuls appearing in the literature, we derive them from first principles. 

\begin{lemma}\label{finitehalfspace}
Let $C\subset\R^n$ be a convex set with dual description 
\begin{equation}\label{dual description}
C=\bigcap_{\vec{s}\in\Sphere^{n-1}}H_{\vec{s}}, 
\end{equation}
where $H_{\vec{s}}=\left\{\vec{x}:\,\langle\vec{s},\vec{x}\rangle\leq\lambda\left(\vec{s}\right)\right\}$ 
for some continuous function $\vec{s}\mapsto\lambda(\vec{s})\in\R$. Then the following hold true: 
\begin{itemize}
\item[a) ] $C$ is compact. 
\item[b) ] For any given $\epsilon>0$, there exists a finite collection of points $\vec{s}_1,\dots,\vec{s}_k\in\Sphere^{n-1}$ for which 
\begin{equation}\label{Sepsilon}
\max_{\vec{x}\in \cap_{i=1}^k H(\vec{s}_i)}d\left(\vec{x},C\right)\leq\epsilon.
\end{equation}
\item[c) ] For every point $\vec{x}\in\partial C$, there exists $s\in\Sphere^{n-1}$ such that 
\begin{equation*}
\left\langle\vec{s},\vec{x}\right\rangle=\max_{\vec{y}\in C}\left\langle\vec{s},\vec{y}\right\rangle
=\lambda\left(\vec{s}\right). 
\end{equation*}
\end{itemize}
\end{lemma}

\begin{proof}
a) Since $C$ is a closed subset of the compact set 
\begin{equation*}
\bigcap_{i=1}^n H_{\vec{e}_i}\cap H_{-\vec{e}_i}, 
\end{equation*}
it is itself compact. 

b) The continuity of $\lambda$ and of the Hausdorff distance imply that the function $s\mapsto d\left(\vec{x},H(s)\right)$ for any fixed $\vec{x}\in\R^{n}$ is continuous. The dual description \eqref{dual description} can therefore be amended by taking the intersection over $\vec{s}$ in a dense subset ${\mathcal S}\subset\Sphere^{n-1}$ only, and since the unit sphere is a separable space, we may take ${\mathcal S}$ to be a countable set $\{\vec{s}_i:\,i\in\N\}$. Thus, we have
\begin{equation}\label{countability}
C=\bigcap_{i=1}^{\infty}H\left(\vec{s}_i\right). 
\end{equation}
Consider the nested sets  
\begin{equation}\label{GH}
G_j:=\cap_{i=1}^j H(\vec{s}_i).
\end{equation}
Since $C$ is compact, we may furthermore order the vectors $\vec{s}_i$ so that $G_j$ is compact for all $j\geq 2n$. Our claim is clearly true if we can establish that 
\begin{equation}\label{wieder}
\max_{\vec{x}\in G_j}d(\vec{x},C)\stackrel{j\rightarrow\infty}{\longrightarrow} 0. 
\end{equation}
Assuming the contrary, there exists $\epsilon>0$, a subsequence $(G_{j_i})_{i\in\N}$ and points  $\vec{y}_{j_i}\in G_{j_i}$ such that 
\begin{equation}\label{hu}
d\left(\vec{y}_{j_i},C\right)\geq\epsilon, \quad\forall\,i\in\N. 
\end{equation}
Since all but finitely many terms of the sequence $(\vec{y}_{j_i})_{i\in\N}$ are contained in the compact set $G_{2n}$, there exists a convergent subsequence $(\vec{z}_k)_{k\in\N}\subseteq(\vec{y}_{j_i})_{i\in\N}$ with limit $\vec{z}_*$. By continuity of the distance function and by virtue of Equation \eqref{hu}, we have  
\begin{equation*}
d\left(\vec{z}_*, C\right)=\lim_{k\rightarrow\infty}d\left(\vec{z}_k, C\right)\geq\epsilon.
\end{equation*}
Since this is in direct contradiction with $\vec{z}_*\in\bigcap_{i\in\N}G_{j_i}=C$, implied by \eqref{countability}, our claim is true. 

c) If $C=\emptyset$, there is nothing to prove. Hence, we may assume that $C$ is nonempty. Without loss of generality, we may furthermore assume that $C$ is shifted so that it contains the origin and $\lambda(\vec{s})\geq 0$ for all $\vec{s}\in\Sphere^{n-1}$. Let $\vec{x}\in\partial C$. Then, by the dual description \eqref{dual description},  
\begin{equation*}
\left\langle\vec{s},\vec{x}\right\rangle\leq\max_{\vec{y}\in C}\left\langle\vec{s},\vec{y}\right\rangle
\leq\lambda\left(\vec{s}\right)
\end{equation*}
for all $\vec{s}\in\Sphere^{n-1}$. Thus, it suffices to prove the existence of $\vec{s}$ such that $\langle\vec{s}, \vec{x}\rangle\geq\lambda(\vec{s})$. If $\vec{x}=\vec{0}$, then we may assume without loss of generality that $C$ has empty interior, as any interior point could otherwise be shifted to the origin. By convexity, $C$ then lies in a lower-dimensional subspace, and it suffices to take $\vec{s}\in\Span(C)^{\perp}$. 
If $\vec{x}\neq\vec{0}$, then $(1+1/k)\vec{x}\notin C$ for all $k\in\N$. By virtue of the Hahn-Banach Theorem there exist unit vectors $\vec{s}_k\in\Sphere^{n-1}$ such that 
\begin{equation}\label{prequel}
\left\langle\vec{s}_k, \vec{x}\right\rangle\geq\frac{\lambda\left(\vec{s}_k\right)}{1+\frac{1}{k}}. 
\end{equation}
$\Sphere^{n-1}$ being compact, we can extract a convergent subsequence and assume without loss of generality that $\vec{s}_k\rightarrow \vec{s}_*\in\Sphere^{n-1}$, as $k\rightarrow\infty$. By continuity of $\lambda$, \eqref{prequel} implies 
\begin{equation*}
\left\langle\vec{s}_*, \vec{x}\right\rangle\geq\lambda\left(\vec{s}_*\right). 
\end{equation*}
\end{proof}

\begin{theorem}\label{finitehalfspace d)}
Let $C$ be as in Lemma \ref{finitehalfspace} and nonempty. Then the points $\vec{s}_i$ in part b) of Lemma \ref{finitehalfspace} can be chosen so that $\vec{x}_i=\arg\max_{\vec{y}\in C}\langle\vec{s}_i, \vec{y}\rangle$ is unique for all $i$, that is, $\vec{x}_i$ are points of strict curvature. 
\end{theorem}

\begin{proof}
By Theorem \ref{convex3}, the set 
\begin{equation*}
{\mathcal S}=\left\{\vec{s}\in\Sphere^{n-1}:\,\arg\max_{\vec{y}\in C}\left\langle\vec{s},\vec{y}\right\rangle \in C_{SE}\right\}
\end{equation*}
has full measure and is therefore everywhere dense in $\Sphere^{n-1}$. Since $\closure({\mathcal S})=\Sphere^{n-1}$ is separable, there exists furthermore a countable subset $\{\vec{s}_i:\,i\in\N\}\subset{\mathcal S}$, also everywhere dense in $\Sphere^{n-1}$, that can be used in the construction. 
\end{proof}

\begin{theorem}\label{222} 
Let $C\subset\R^n$ be nonempty compact convex. The for all $\epsilon>0$ there exist finitely many points of strict curvature $\vec{x}_1,\dots,\vec{x}_k\in C_{SE}$ such that  
\begin{equation*}
\max_{\vec{x}\in C}d\left(\vec{x},\conv\left(\vec{x}_1,\dots,\vec{x}_k\right)\right)\leq\epsilon.
\end{equation*}
\end{theorem}

\begin{proof}
Let $\{\vec{x}_1,\dots,\vec{x}_k\}\subset C_{SE}$ be an $\epsilon$-net on the set $C_E$ of extreme points of $C$, that is, $\vec{x}_i$ are chosen so that 
\begin{equation*}
\min_i\left\|\vec{z}-\vec{x}_i\right\|\leq\epsilon
\end{equation*}
for all $\vec{z}\in C_E$. The existence of such an $\epsilon$-net is established as follows: $C$ being compact, $\closure(C_{SE})$ is a compact set too, and by the Heine-Borel Theorem, we can extract a finite covering by Euclidean balls of radius $\epsilon/2$ around points $\vec{y}_i\in\closure(C_{SE})$, which by Proposition \ref{lala d)} b) is also a covering of $C_E$, 
\begin{equation*}
\bigcup_{i=1}^k \B_{\epsilon/2}\left(\vec{y}_i\right)\supset C_{E}. 
\end{equation*}
Next, for all $i$ choose $\vec{x}_i\in C_{SE}$ within distance $\epsilon/2$ of $\vec{y}_i$. It then follows from the triangular inequality that $\{\vec{x}_1,\dots,\vec{x}_k\}$ is the required $\epsilon$-net. 

By the Theorems of Minkowski\footnote{This theorem says that a convex compact set in $\R^n$ is equal to the convex hull of its extreme points. The generalization of this result to arbitrary topological vector spaces is the Krein-Milman Theorem \cite{krein, milman}.} \cite{minkowski} and Carath\'eodory\footnote{This result says that if $K=\conv(X)$ for some $X\subset\R^n$, then every point in $K$ is a convex combination of at most $n+1$ elements of $X$.} \cite{caratheodory}, any point $\vec{x}\in C$ can be written as a convex combination $\vec{x}=\xi_1\vec{z}_1+\dots+\xi_m\vec{z}_m$ of $m\leq n+1$ extreme points $z_j\in C_E$, and by construction of the $\epsilon$-net, it is then possible to choose $1\leq i_j\leq k$ such that $\|\vec{z}_j-\vec{x}_{i_j}\|\leq\epsilon$ for all $j$. Using the triangular inequality once again, we find that 
\begin{align*}
d\left(\vec{x}, \conv\left(\vec{x}_1,\dots,\vec{x}_k\right)\right)&\leq
d\left(\vec{x}, \conv\left(\vec{x}_{i_1},\dots,\vec{x}_{i_m}\right)\right)\\
&\leq
d\left(\xi_1\vec{z}_1+\dots+\xi_m\vec{z}_m, \xi_1\vec{x}_{i_1}+\dots+\xi_{m}\vec{x}_{i_m}\right)\leq\epsilon,
\end{align*}
as claimed.
\end{proof}

The next result can be seen as an intuitive template for Theorem \ref{convergence2SET} free of large deviations complications.

\begin{theorem}\label{convex1}
Let $C$ be a nonempty convex compact subset of $\R^n$ with dual description \eqref{finitehalfspace}, and let $C^1,C^2,\dots$ be convex compact subsets of $\R^n$ such that for all linear functionals $f:\R^n\rightarrow\R$ it is true that 
\begin{equation}\label{assumption assumption}
\max_{\vec{p}\in C^n}f\left(\vec{p}\right)\stackrel{n\rightarrow\infty}{\longrightarrow}
\max_{\vec{p}\in C}f\left(\vec{p}\right).
\end{equation}
Then $d(C^n, C)\rightarrow 0$. 
\end{theorem}

\begin{proof}
Upon translation and rescaling we may assume without loss of generality that $\{0\}\in C\subset\B_{1}(0)$. Let $\epsilon>0$ be a given small number. By the dual description of $C$ and Lemma \ref{finitehalfspace} there exist finitely many vectors $\vec{s}_1,\dots,\vec{s}_k\in\Sphere^{n-1}$ such that 
\begin{equation}\label{salami}
\sup_{\vec{x}\in C_{\epsilon}}\,d\left(\vec{x},C\right)<\epsilon, 
\end{equation}
where $C_{\epsilon}=\cap_{i=1}^{k}\left\{\vec{x}:\,\left\langle\vec{s}_i,\vec{x}\right\rangle\leq\lambda_{s_i}\right\}$ and $\lambda_s=\arg\max_{\vec{x}\in C}\left\langle\vec{s}, \vec{x}\right\rangle\geq 0$. Because of Assumption \eqref{assumption assumption}, there exists $n_0\in\N$ such that for all $n\geq n_0$, 
\begin{equation*}
\max_{\vec{x}\in C^n}\left\langle\vec{s}_i, \vec{x}\right\rangle\leq(1+\epsilon)\lambda_{\vec{s}_i},\quad
(i=1,\dots,k), 
\end{equation*}
and hence, $C^n\subseteq(1+\epsilon)C_{\epsilon}$. Therefore, for all $n\geq n_0$, 
\begin{align}
\sup_{\vec{x}\in C^n}\,d(\vec{x},C)&\leq\sup_{\vec{x}\in(1+\epsilon)C_{\epsilon}}\,d(\vec{x},C)\nonumber\\
&=(1+\epsilon)\sup_{\vec{x}\in C_{\epsilon}}\,d\left(\vec{x},\frac{1}{1+\epsilon}C\right)\nonumber\\
&\leq(1+\epsilon)\left(\sup_{\vec{x}\in C_{\epsilon}}\,d\left(\vec{x},C\right)+\sup_{\vec{x}\in C}\,d\left(\vec{x},\frac{1}{1+\epsilon}C\right)\right)\nonumber\\
&\leq 2(1+\epsilon)\epsilon\label{first bound}
\end{align}
where we used \eqref{salami} and $d(C,1/(1+\epsilon)C_{\epsilon})<\epsilon$ to arrive at the last inequality. 

Next, let us choose points $\vec{x}_1,\dots,\vec{x}_k$ as in Theorem \ref{222}. Proposition \ref{lala d)} allows us to choose scoring functions $S_{\vec{x}_i}$ and $S_1,\dots,S_{\ell}$ such that $\vec{x}_i\in C_i\subset\B_{\epsilon}(\vec{x}_i)$ and $C_i^j$ is compact for all $j\in\N$, where 
\begin{align*}
C_i&:=\left\{\vec{x}:\,f_{S_{\vec{x}_i}}\left(\vec{x}\right)\geq\max_{\vec{x}\in C}f_{S_{\vec{x}_i}}\left(\vec{x}\right),\, f_{S_j}\left(\vec{x}\right)\leq\max_{\vec{x}\in C}f_{S_j}\left(\vec{x}\right), (j=1,\dots,\ell)\right\},\\
C_i^j&:=\left\{\vec{x}:\,f_{S_{\vec{x}_i}}\left(\vec{x}\right)\geq\max_{\vec{x}\in C^j}f_{S_{\vec{x}_i}}\left(\vec{x}\right),\, f_{S_j}\left(\vec{x}\right)\leq\max_{\vec{x}\in C^j}f_{S_j}\left(\vec{x}\right), (j=1,\dots,\ell)\right\}.
\end{align*}
Taking $\vec{x}_{ij}$ to be a maximizer of $\arg\max_{\vec{x}\in C^j}f_{S_{\vec{x}}}$, we find  $\vec{x}_{ij}\in C^j\cap C_i^j\neq\emptyset$, and by \eqref{assumption assumption} applied to $f_{S_{\vec{x}_i}}$ and $f_{S_j}$ $(j=1,\dots,k)$ and the continuity of the Hausdorff distance, we have 
\begin{equation*}
\limsup_{n\rightarrow\infty}\,\left\|\vec{x}_i-\vec{x}_{ij}\right\|\leq\epsilon. 
\end{equation*}
By the triangular inequality, this implies 
\begin{equation*}
\limsup_{n\rightarrow\infty} \max_{\vec{x}\in C}d(\vec{x}, C^n)\leq 2\epsilon,\quad
\forall\,\epsilon>0, 
\end{equation*}
and since this is true for all $\epsilon$ rational, the Theorem follows. 
\end{proof}

The final result shows among other things that all results presented above are all applicable to $SET$, defined as in \eqref{SET}. Recall the notations $\lambda(S)$ for the Chv\`atal-Sankoff limit \eqref{chvatal} of a scoring function $S$, and $f_S$, the linear functional associated with $S$ as defined in \eqref{defined in}. 

\begin{lemma}\label{lala}\hspace{1cm}
\begin{itemize}
\item[a) ] The function $S\mapsto\lambda(S)$ is continuous. 
\item[b) ] $SET$ is nonempty convex compact. 
\item[c) ] For every $\vec{x}\in\partial SET$ there exists $S_{\vec{x}}\neq 0$ such that $\{\vec{y}:\,f_{S_{\vec{x}}}(\vec{y})=\lambda(S_{\vec{x}})\}$ is a tangent plane to $SET$ supported at $\vec{x}$. 
\item[d) ] $\max_{\vec{x}\in SET}f_S\left(\vec{x}\right)=\lambda(S)$ holds true for all scoring functions $S$. 
\end{itemize}
\end{lemma}

\begin{proof}
a) It suffices to show that for any two scoring functions $S$ and $T$ the following inequality holds true,  
\begin{equation*}
\lambda_n(S)-\lambda_n(T)\leq 2\left\|T-S\right\|_\infty. 
\end{equation*}
For any pair of letters $\mathfrak{c,d}\in\mathcal{A}^*$ we have 
\begin{equation*}
|S(\mathfrak{c,d})-T(\mathfrak{c,d})|\leq\|S-T\|_\infty=\max_{\mathfrak{a,b}\in{\mathcal A}^*}
\left|S(\mathfrak{a,b})-T(\mathfrak{a,b})\right|. 
\end{equation*}
Further, since no optimal alignment of two strings of length $n$ contains any aligned pair of gaps, there are at most $2n$ aligned letter pairs, so that the triangular inequality implies 
\begin{equation*}
\left|L_n(S)-L_n(T)\right|\leq 2n\|S-T\|_\infty. 
\end{equation*}
The claim now follows by dividing by $n$ and taking expectations. Lemma \ref{finitehalfspace} now applies because $SET$ has the dual description \eqref{SET}, where $S$ can be restricted to the case where $f_S$ is a unit vector, since for all $\tau>0$ and pairs of strings $(x,y)$, we have $(\tau S)_{\pi}(x,y)=\tau S_\pi(x,y)$, whence $\lambda(\tau S)=\tau\lambda(S)$ and $H(\tau S)=H(S)$. 

b) By Part a), Lemma \ref{finitehalfspace} is applicable to $C=SET$. This shows that $SET$ is convex compact. It remains to prove that it contains at least one point. Consider the random sequences $X,Y$ introduced in Section \ref{overview}, and for all $n\in\N$ let the first $n$ letters be aligned by $\pi_n$, defined as follows,  
\begin{equation*}
\begin{array}{c||c|c|c|c|c|c|c|c}
X&X_1&\mathfrak{G}&X_2&\mathfrak{G}&X_3&\dots&X_n&\mathfrak{G}\\\hline
Y&\mathfrak{G}&Y_1&\mathfrak{G}&Y_2&\mathfrak{G}&\dots&\mathfrak{G}&Y_n
\end{array},
\end{equation*}
where $\mathfrak{G}$ denotes a gap. For all $n\in\N$, the expected empirical distribution of aligned letters is given by the vector $\vec{x}$ with entries defined as follows, 
\begin{equation*}
x_{\mathfrak{a,b}}=\begin{cases}\prob[X_1=\mathfrak{a}],\quad&\text{if }\mathfrak{a}\neq\mathfrak{G}, \mathfrak{b=G},\\
\prob[Y_1=\mathfrak{b}]\quad&\text{if }\mathfrak{a}=\mathfrak{G}, \mathfrak{b}\neq\mathfrak{G},\\
0\quad&\text{otherwise}. 
\end{cases}
\end{equation*}
For every scoring function $S$ this is a legitimate, albeit suboptimal, alignment. Therefore, we have 
\begin{equation*}
f_S\left(\vec{x}\right)=\frac{1}{n}\expect\left[S_{\pi_n}(X,Y)\right]
\stackrel{\eqref{new old}}{\leq}\frac{1}{n}\expect\left[L_n(S)\right]
\stackrel{\eqref{increasing}}{\leq}\lambda(S).
\end{equation*}
By \eqref{SET}, this shows that $\vec{x}\in SET$.  

c) This follows from Lemma \ref{finitehalfspace}, which is applicable by Part a). 

d) Note that it follows from \eqref{H} and \eqref{SET} that $f(\vec{x})\leq\lambda(S)$ for all $\vec{x}$. It suffices thus to show that the hyperplane 
\begin{equation*}
{\mathcal T}:=\left\{\vec{x}:\,f_S\left(\vec{x}\right)=\lambda(S)\right\}
\end{equation*}
has nonempty intersection with $SET$. Assuming the contrary, and using the construction of \eqref{countability} and \eqref{GH}, there exists $j\in\N$ such that ${\mathcal T}\cap G_j=\emptyset$. By continuity, there then exists $\delta>0$ such that 
\begin{equation}\label{absurdity}
{\mathcal T}^{\delta}\cap G^{\delta}_j=\emptyset, 
\end{equation}
where 
\begin{align*}
{\mathcal T}^{\delta}&:=\left\{\vec{x}:\,f_S\left(\vec{x}\right)\geq\lambda(S)-\delta\right\},\\
G^{\delta}_j&:=\bigcap_{i=1}^j H^{\delta}\left(S_i\right),\\
H^{\delta}\left(S_i\right)&:=\left\{\vec{x}:\,f_{S_i}\left(\vec{x}\right)\leq\lambda\left(S_i\right)+\delta\right\}. 
\end{align*}
By \eqref{AAA} and the almost sure convergence of $L_n(S_i)/n$ to $\lambda(S_i)$, there exists almost surely $n_0\in\N$ such that for all $n\geq n_0$,
\begin{equation*}
\max_{\vec{x}\in SET^n}\,f_{S_i}\left(\vec{x}\right)\leq\lambda\left(S_i\right)+\delta,\quad (i=1,\dots,j), 
\end{equation*}
so that $SET^n\subset G^{\delta}_j$, and further, 
\begin{equation*}
\max_{\vec{x}\in SET^n}\,f_S\left(\vec{x}\right)\geq\lambda(S)-\delta, 
\end{equation*}
so that $\emptyset\neq {\mathcal T}^{\delta}\cap SET^n\subset{\mathcal T}^{\delta}\cap G^{\delta}_j$. Since this is in direct contradiction with \eqref{absurdity}, the claim holds true. 
\end{proof}

\section{Appendix: Large Deviations}\label{appendix}

Recall the notations $L_S(x_1\dots x_i,y_1\dots y_j)$, $L_n(S)=L_S(X_1\dots X_n, Y_1\dots Y_n)$ and $\lambda_n(S)=\expect[L_n(S)]/n$ introduced in Section \ref{overview}, and the fact that $\lambda_n(S)\rightarrow\lambda(S)$ mentioned in \eqref{chvatal}. In this appendix we will show a stronger result that quantifies the convergence rate as being of order $\Bigo(\sqrt{\ln n/n})$. For this purpose, we introduce the following notation, 
\begin{align*}
\|S\|_{\delta}&=\max_{{\mathfrak c},{\mathfrak d},{\mathfrak e}\in{\mathcal A}^*}\left|S\left({\mathfrak c},{\mathfrak d}\right)-S\left({\mathfrak c},{\mathfrak e}\right)\right|,\\
\|S\|_{\infty}&=\max_{{\mathfrak c},{\mathfrak d}\in{\mathcal A}^*}\left|S\left({\mathfrak c},{\mathfrak d}\right)\right|,
\end{align*}

\begin{lemma}\label{change}
Let $x=x_1\dots x_m$ and $y=y_1\dots y_n$ be two given strings with letters from the alphabet ${\mathcal A}$, and let $S$ be a given scoring function. Let further $\hat{x}\in{\mathcal A}$, and consider two amendments of string $x$, $x^{[i]}=x_1\dots x_{i-1}\,\hat{x}\,x_{i+1}\dots x_m$, obtained by replacing an arbitrary letter $x_i$ by $\hat{x}$, and $x^{[+]}=x_1\dots x_m\,\hat{x}$, obtained by extending $x$ by a letter $\hat{x}$. Then the following hold true, 
\begin{align}
\left|L_S(x^{[i]},y)-L_S(x,y)\right|&\leq\|S\|_{\delta},\label{first claim}\\
\left|L_S(x^{[+]},y)-L_S(x,y)\right|&\leq\|S\|_{\infty}.\label{second claim}
\end{align}
\end{lemma}

\begin{proof}
Let $\pi$ be an optimal alignment of $x$ and $y$, so that $S_{\pi}(x,y)=L_S(x,y)$, and denote the letter with which $x_i$ is aligned under $\pi$ by ${\mathfrak a}\in{\mathcal A}^*$. Then 
\begin{equation*}
L_S(x^{[i]},y)\geq S_{\pi}(x^{[i]},y)=S_{\pi}(x,y)-S(x_i,{\mathfrak a})+S(\hat{x},{\mathfrak a})
\geq L_S(x,y)-\|S\|_{\delta}. 
\end{equation*}
Applying the identical argument to an optimal alignment of $x^{[i]}$ and $y$, we obtain the analogous inequality 
\begin{equation*}
L_S(x,y)\geq L_S(x^{[i]},y)-\|S\|_{\delta}, 
\end{equation*}
so that \eqref{first claim} follows. 

For the second claim, let us use an optimal alignment $\pi$ of $x$ and $y$ to construct an alignment $\pi^{[+]}$ of $x^{[+]}$ and $y$ by appending an aligned pair of letters $(\hat{x},\mathfrak{G})$, where 
$\mathfrak{G}$ denotes a gap. Then we have 
\begin{equation*}
L_S(x^{[+]},y)\geq S_{\pi^{[+]}}(x^{[+]},y)=S_{\pi}(x,y)+S(\hat{x},\mathfrak{G})\geq L_S(x,y)-\|S\|_{\infty}. 
\end{equation*}
Conversely, we can amend an optimal alignment $\tilde{\pi}^{[+]}$ of $x^{[+]}$ and $y$ to become a valid alignment $\tilde{\pi}$ of $x$ and $y$ by cropping the last pair of aligned letters, $(\hat{x},{\mathfrak a})$. We then have 
\begin{equation*}
L_S(x,y)\geq S_{\tilde{\pi}}(x,y)=S_{\tilde{\pi}^{[+]}}(x^{[+]},y)-S(\hat{x},{\mathfrak a})\geq L_S(x^{[+]},y)-\|S\|_{\infty},
\end{equation*}
thus establishing \eqref{second claim}. 
\end{proof}

\begin{lemma}\label{rateoflambdan}
The convergence of $\lambda_n(S)$ to $\lambda(S)$ is governed by the inequality 
\begin{equation}\label{difference}
\lambda_n(S)\leq\lambda(S)\leq\lambda_n(S)+ c_n\|S\|_{\delta}\frac{\sqrt{\ln n}}{\sqrt{n}}+\frac{2\|S\|_{\infty}}{n},\quad\forall\,n\in\N,
\end{equation}
where 
\begin{equation*}
c_n:=\sqrt{\frac{2\ln3+2\ln(n+2)}{\ln(n)}}.
\end{equation*}
\end{lemma}

Note that $c_n$ tends to $\sqrt{2}$ when $n\rightarrow\infty$, so that it effectively acts as a constant. 

\begin{proof} 
See \cite{lcsMontecarlo}. 
\end{proof}

\begin{lemma}[McDiarmid's Inequality \cite{mcDiarmid}]\label{Azuma} 
Let $Z_1,Z_1,\dots,Z_m$ be i.i.d.\ random variables that take values in a set $D$, and let $g:D^m\rightarrow\R$ be a function of $m$ variables with the property that 
\begin{equation*}
\max_{i=1,\dots,m}\sup_{z\in D^m, \hat{z}_i\in D}\left|g(z_1,\dots,z_m)-g(z_1,\dots,\hat{z}_i,\dots,z_m)\right|\leq C. 
\end{equation*}
Thus, changing a single argument of $g$ changes its image by less than a constant $C$. Then the following bounds hold, 
\begin{align*}
\prob\left[g(Z_1,\dots,Z_m)-\expect[g(Z_1,\dots,Z_m)]\geq \epsilon\times m\right]&\leq\exp\left\{-\frac{2\epsilon^2 m}{C^2}\right\},\\
\prob\left[\expect\left[g(Z_1,\dots,Z_m)\right]-g(Z_1,\dots,Z_m)\geq \epsilon\times m\right]&\leq\exp\left\{-\frac{2\epsilon^2 m}{C^2}\right\}.
\end{align*}
\end{lemma}

\begin{proof}
A consequence of the Azuma-Hoeffding Inequality, see \cite{mcDiarmid}. 
\end{proof}

\begin{theorem}\label{Azuma etc}
For fixed $\epsilon>0$ and scoring function $S$ there exists $K>0$ and $n_{\epsilon}\in\N$ such that 
\begin{align}
\prob\left[\frac{L_n(S)}{n}\geq\lambda(S)+\epsilon\right]&\leq\e^{-Kn},\quad\forall\,n\in\N,\label{cor1}\\
\prob\left[\frac{L_n(S)}{n}\leq\lambda_n(S)-\epsilon\right]&\leq\e^{-Kn},\quad\forall\,n\in\N,\label{cor2}\\
\prob\left[\frac{L_n(S)}{n}\leq\lambda(S)-\epsilon\right]&\leq\e^{-Kn},\quad\forall\,n\gg n_{\epsilon}.\label{cor3}
\end{align}
\end{theorem}

\begin{proof}
We know from Lemma \ref{change} that 
\begin{equation*}
g(X_1,\dots,X_n,Y_1,\dots,Y_n)=S(X_1\dots X_n, Y_1\dots Y_n)=L_n(S)
\end{equation*}
satisfies the assumptions of Lemma \ref{Azuma} with $m=2n$ and $C=\|S\|_{\delta}$. McDiarmid's Inequality therefore shows 
\begin{align}
\prob\left[\frac{L_n(S)}{n}\geq\lambda_n(S)+\epsilon\right]&=
\prob\left[L_n(S)\geq\expect\left[L_n(S)\right]+\frac{\epsilon}{2}\times 2n\right]\nonumber\\
&\leq\exp\left\{-\frac{\epsilon^2}{\|S\|_{\delta}^2}\times n\right\},\label{thu1}
\end{align}
and similarly, 
\begin{equation}\label{thu2}
\prob\left[\frac{L_n(S)}{n}\leq\lambda_n(S)-\epsilon\right]\leq\exp\left\{-\frac{\epsilon^2}{\|S\|_{\delta}^2}\times n\right\}.
\end{equation}
Claim \eqref{cor2} therefore holds with $K=\frac{\epsilon^2}{4\|S\|_{\delta}^2}$.

Furthermore, Lemma \ref{rateoflambdan} established that 
\begin{equation}\label{difference2}
\lambda_n(S)\leq\lambda(S)\leq\lambda_n(S)+ c_n\|S\|_{\delta}\frac{\sqrt{\ln n}}{\sqrt{n}}+\frac{2\|S\|_{\infty}}{n},\quad\forall\,n\in\N,
\end{equation}
holds, where $c_n=\sqrt{2\ln3+2\ln(n+2)}/\sqrt{\ln(n)}$.  Using the first inequality from \eqref{difference2} in conjunction with \eqref{thu1}, we find 
\begin{equation*}
\prob\left[\frac{L_n(S)}{n}\geq\lambda(S)+\epsilon\right]\leq\prob\left[\frac{L_n(S)}{n}\geq\lambda_n(S)+\epsilon\right]\leq\exp\left\{-\frac{\epsilon^2}{\|S\|_{\delta}^2}\times n\right\},
\end{equation*}
which shows that Claim \eqref{cor1} holds with $K=\frac{\epsilon^2}{4\|S\|_{\delta}^2}$.

Using now the second inequality from \eqref{difference2} in conjunction with \eqref{thu2}, we find 
\begin{align*}
\prob\left[\frac{L_n(S)}{n}\leq\lambda(S)-\epsilon\right]
&\leq\prob\left[\frac{L_n(S)}{n}\leq\lambda_n(S)-\left(\epsilon-c_n\|S\|_{\delta}\frac{\sqrt{\ln n}}{\sqrt{n}}-\frac{2\|S\|_{\infty}}{n}\right)\right]\\
&\leq\exp\left\{-\frac{\left(\epsilon-c_n\|S\|_{\delta}\frac{\sqrt{\ln n}}{\sqrt{n}}-\frac{2\|S\|_{\infty}}{n}\right)^2}{\|S\|_{\delta}^2}\times n\right\}\\
&\leq\exp\left\{-\frac{\epsilon^2}{4\|S\|_{\delta}^2}\times n\right\},\quad\forall\,n\geq n_{\epsilon},
\end{align*}
where $n_{\epsilon}\in\N$ is chosen large enough to satisfy
\begin{equation*}
\epsilon-c_n\|S\|_{\delta}\frac{\sqrt{\ln n}}{\sqrt{n}}-\frac{2\|S\|_{\infty}}{n}>\frac{\epsilon}{2},\quad\forall\,n\geq n_{\epsilon}.
\end{equation*}
This shows that \eqref{cor3} holds for $K=\frac{\epsilon^2}{4\|S\|_{\delta}^2}$.
\end{proof}

\section{Acknowledgements} 

The authors wish to thank the Engineering and Physical Sciences Research Council (EPSRC) and the Institute of Mathematics and its Applications (IMA) and Pembroke College Oxford for generous financial support, Endre S\"uli for supporting their small grant application to the IMA, Pembroke College for granting a College Associateship to Heinrich Matzinger, and the Oxford Mathematical Institute for hosting him during his Oxford 
visit. 

\bibliographystyle{plain}
\bibliography{bio12}

\end{document}